\documentclass[10pt]{article}

\usepackage[a4paper, total={6.5in, 10in}, includehead]{geometry}
\usepackage[parfill]{parskip}
\usepackage{xcolor}
\usepackage{microtype}
\usepackage{orcidlink}

\usepackage[runin]{abstract}
\usepackage[twoside]{fancyhdr}
\usepackage[explicit]{titlesec}

\usepackage[T1]{fontenc}
\usepackage[cal=esstix,
            scr=boondoxupr,
            frak=euler,
            bb=libus,
            scaled=0.95]{mathalpha}
\usepackage[OT1, small,
            euler-hat-accent,
            euler-digits]{eulervm}
\usepackage{Baskervaldx}
\usepackage{pifont}

\usepackage[leqno]{amsmath}
\usepackage{amssymb}
\usepackage{amsthm}
\usepackage{mathtools,thmtools}
\usepackage{stmaryrd}
\usepackage{xfrac,xparse,xspace}
\usepackage{stackrel,scalerel}

\usepackage[shortlabels]{enumitem}

\usepackage[english]{babel}
\usepackage{csquotes}
\usepackage[backend=biber,
            style=alphabetic,
            hyperref=auto,
            useprefix=false,
            maxbibnames=99,
            maxalphanames=99,
            maxcitenames=99]{biblatex}
\usepackage{hyperref}
\usepackage{cleveref}
\usepackage{url}
\usepackage{crossreftools}

\usepackage{graphicx}
\usepackage{tikz,tikz-cd,tikz-imagelabels}
\usepackage{float}
\usepackage[font=small,
            labelfont=bf,
            labelsep=period,
            format=plain,
            singlelinecheck=false]{caption}

\hypersetup{
    hypertexnames = false,
    bookmarksdepth = 2,
    bookmarksopen = true,
    colorlinks,
    linkcolor = ., citecolor = teal, urlcolor = .,
    pdfstartview={XYZ null null 1}}

\DeclareSortingNamekeyTemplate{
  \keypart{\namepart{family}}
  \keypart{\namepart{prefix}}
  \keypart{\namepart{given}}
  \keypart{\namepart{suffix}}
}
\renewbibmacro*{begentry}{\midsentence}
\makeatletter
\AtBeginDocument{\toggletrue{blx@useprefix}}
\makeatother

\declaretheoremstyle[
    spaceabove=\parsep, spacebelow=\parsep,
    headfont=\bfseries, notefont=\normalfont, bodyfont=\itshape,
    headpunct={.}, notebraces={}{}, postheadspace={ },
    ]{basic-theorem}
\declaretheoremstyle[
    spaceabove=\parsep, spacebelow=\parsep,
    headfont=\bfseries, notefont=\normalfont, bodyfont=\normalfont,
    headpunct={.}, notebraces={(}{)}, postheadspace={ },
    ]{basic-definition}
\declaretheoremstyle[
    spaceabove=\parsep, spacebelow=\parsep,
    headfont=\itshape, notefont=\normalfont, bodyfont=\normalfont,
    headpunct={.}, notebraces={(}{)}, postheadspace={ },
    ]{basic-remark}

\theoremstyle{basic-theorem}
\newtheorem{keytheorem}{Theorem}[section]

\newtheorem{keyconjecture}[keytheorem]{Conjecture}
\newtheorem{theorem}{Theorem}[section]
\newtheorem{corollary}[theorem]{Corollary}
\newtheorem{lemma}[theorem]{Lemma}

\theoremstyle{basic-definition}
\newtheorem{notation}[theorem]{Notation}

\theoremstyle{basic-remark}

\titleformat{\section}{\normalfont\bfseries\large\centering}{\S\,\thesection}{0.5em}{#1}{}
\crefname{section}{\S\kern -1pt}{\S\S \kern -1pt}
\newcommand{\fakesection}[1]{%
    \par\refstepcounter{section}%
    \sectionmark{#1}%
    \addcontentsline{toc}{section}{\protect\numberline{\thesection}#1}%
}

\titleformat{\subsection}[runin]{\normalfont\bfseries}{\S\,\thesubsection}{0.5em}{#1.}{}
\titlespacing{\subsection}{0pt}{1em}{\wordsep}[]
\crefname{subsection}{\S \kern -1pt}{\S\S \kern -1pt}

\imagelabelset{
    coarse grid color = red,
    fine grid color = gray,
    image label font = \sffamily\bfseries\scriptsize,
    image label distance = 2mm,
    image label back = none,
    image label text = black,
    coordinate label font = \normalfont\scriptsize,
    coordinate label distance = 2mm,
    coordinate label back = none,
    coordinate label text = black,
    annotation font = \normalfont\scriptsize,
    arrow distance = 0,
    border thickness = 0.2pt,
    arrow thickness = 0.1pt,
    tip size = 0.5mm,
    outer dist = 0.5cm,
}

\newcommand{\xmark}{\ding{55}}%

\newcommand{\ul}{\underline}
\newcommand{\ol}{\overline}
\newcommand{\orth}{{}^\perp}

\newcommand{\bbC}{\mathbb{C}}
\newcommand{\bbR}{\mathbb{R}}
\newcommand{\bbZ}{\mathbb{Z}}
\newcommand{\bbN}{\mathbb{N}}
\newcommand{\bbk}{\mathbb{k}}
\newcommand{\KK}{\mathbf{K}}
\newcommand{\HH}{\textnormal{H}}
\newcommand{\OO}{\mathscr{O}}

\newcommand{\into}{\hookrightarrow}
\newcommand{\onto}{\twoheadrightarrow}
\newcommand{\isoto}{\mathrel{\ooalign{
     $\to$\cr
     \hidewidth\raise.3em\hbox{$\scaleobj{.7}{\sim}\mkern7mu$}\cr
    }}
}
\newcommand{\leftrarrows}{\mathrel{\raise.75ex\hbox{\oalign{%
  $\scriptstyle\leftarrow$\cr
  \vrule width0pt height.5ex$\hfil\scriptstyle\relbar$\cr}}}}
\newcommand{\lrightarrows}{\mathrel{\raise.75ex\hbox{\oalign{%
  $\scriptstyle\relbar$\hfil\cr
  $\scriptstyle\vrule width0pt height.5ex\smash\rightarrow$\cr}}}}
\newcommand{\Rrelbar}{\mathrel{\raise.75ex\hbox{\oalign{%
  $\scriptstyle\relbar$\cr
  \vrule width0pt height.5ex$\scriptstyle\relbar$}}}}

\makeatletter
\def\leftrightarrowsfill@{\arrowfill@\leftrarrows\Rrelbar\lrightarrows}
\newcommand{\xleftrightarrows}[2][]{\ext@arrow 3399\leftrightarrowsfill@{#1}{#2}}
\makeatother

\newcommand{\curlymod}{\ensuremath{\mathcal{m\kern -0.8pt o \kern -0.5pt d}}}
\newcommand{\curlyExt}{\ensuremath{\mathcal{E\kern -0.5pt x\kern -1pt t}}}
\newcommand{\curlyHom}{\ensuremath{\mathcal{H\kern -2pt o\kern -1pt m}}}
\newcommand{\curlyCoh}{\ensuremath{\mathcal{C\kern -2pt o\kern -1pt h\kern -1pt}}}
\newcommand{\curlycoh}{\ensuremath{\mathcal{c\kern -1pt o\kern -1pt h\kern -1pt}}}
\newcommand{\fl}{\textnormal{fl}}

\newcommand{\rmod}{\mathop{\curlymod}\!}
\newcommand{\rflmod}{\mathop{\fl\kern 0.5pt\curlymod}}
\newcommand{\Coh}{\mathop\curlyCoh}
\newcommand{\coh}{\mathop\curlycoh}
\DeclareMathOperator{\Dm}{\mathbf{D}^\mathbf{0}\kern -1.5pt}
\DeclareMathOperator{\Db}{\mathbf{D}^\textnormal{b}\kern -1pt}
\DeclareMathOperator{\Dfl}{\mathbf{D}^\fl\kern -1pt}

\DeclareMathOperator{\Spec}{\textnormal{Spec}}
\DeclareMathOperator{\Hom}{\textnormal{Hom}}
\DeclareMathOperator{\RHom}{\ensuremath{\mathbf{R}\!\Hom}}
\DeclareMathOperator{\End}{\textnormal{End}}
\DeclareMathOperator{\Ext}{\textnormal{Ext}}
\DeclareMathOperator{\Aut}{\textnormal{Aut}}
\DeclareMathOperator{\Stab}{\textnormal{Stab}}
\DeclareMathOperator{\sheafExt}{\curlyExt}
\DeclareMathOperator{\sheafHom}{\curlyHom}

\renewcommand*{\ker}{\mathop\textnormal{ker}}
\renewcommand*{\dim}{\mathop\textnormal{dim}}

\newcommand{\SKMS}{\mathcal{M}_\textnormal{SK}}
\newcommand{\StabX}{\Stab(\Dm X)}

\DeclareMathOperator{\Pic}{\textnormal{Pic}}

\DeclareMathOperator{\tors}{\textnormal{tors}}
\DeclareMathOperator{\torf}{\textnormal{torf}}
\DeclareMathOperator{\Supp}{\textnormal{Supp}}

\begin{document}
\title{\huge Simple complexes on a flopped curve}
\author{Parth Shimpi \orcidlink{0009-0008-2066-7474}}

\date{}
\maketitle

\begin{abstract}
    Studying crepant blow-ups of (compound) du Val singularities, we classify
    complexes of coherent sheaves which admit no negative
    self-extensions --- such a complex, up to flops and mutation equivalences,
    must either be (1) a module over a derived--equivalent algebra, or (2) a
    two--term extension of a coherent sheaf by skyscraper sheaves, or (3) a
    direct sum of shifts of skyscrapers.

    This translates into classifications of bricks, spherical objects, stability
    conditions, and algebraic t-structures in the local derived category; the lists
    populated by the homological minimal programme turn out exhaustive. We deduce
    that the Bridgeland stability manifold is connected, and that all basic
    tilting complexes on the variety (equivalently on the \(g-\)tame algebras
    derived-equivalent to it) are related by shifts and iterated mutation.
\end{abstract}

\medskip
\fakesection{Introduction}

Proper birational morphisms \(\pi:X\to Z\) that crepantly contract an
irreducible rational curve \(C\subset X\) to a point \(\mathbf{0}\in Z\) have
for decades piqued the curiosities of algebraic and symplectic geometers alike,
and the advent of non-commutative and homological techniques in dimensions \(2\)
and \(3\)
\cite{kapranovKleinianSingularitiesDerived,bridgelandFlopsDerivedCategories,vandenberghThreedimensionalFlopsNoncommutative,wemyssFlopsClustersHomological}
has pushed into limelight the problem of understanding derived equivalences
involving \(X\).

Any solution to the problem must involve a classification of \emph{spherical
objects}, i.e.\ complexes which generate an autoequivalence
of \(\Db X\) mirroring a Dehn twist. These are objects
whose self-extensions resemble the singular cohomology of a sphere
\cite{seidelBraidGroupActions}, and this condition is sufficiently
restrictive that classification attempts have found success in specific cases
where the geometry of \(X\) and its mirror is well-understood \cite{ishiiAutoequivalencesDerivedCategories,keatingSymplectomorphismsSphericalObjects2024}.

Walking the thin line that separates courage from na\"ivet\'e, we forget the
cohomological restriction and instead ask if it is possible to classify \emph{all}
complexes supported on \(C\) that admit no negative self-extensions.

\subsection*{Complexes, classified} Let \(Z=\Spec R\) above be a complete local
\(\bbC\)-scheme of dimension \(2\) (or \(3\)), with an at worst canonical
(resp.\ terminal) singularity at the closed point \(\mathbf{0}\). Thus the
map \(\pi\) models a partial resolution of a (compound) du Val singularity,
obtained by blowing down all curves but one in some minimal model.

Certain wall--crossing rules (realising a flop when \(\dim Z=3\)) yield another
crepant partial resolution \(W\to Z\) that is derived--equivalent to \(X\). To
study the pair of varieties, the homological minimal model programme
\cite{wemyssFlopsClustersHomological,donovanStringyKahlerModuli} furnishes a
\((\bbZ/N\bbZ)\)-indexed family of \emph{modification algebras}
\(\{\Lambda_0,...,\Lambda_{N-1}\}\) with equivalences
\begin{equation}
    \label{eqn:allfunctors} \tag{\(\ast\)}
    \Db X\xrightarrow{\quad \text{VdB}\quad} \Db \Lambda_0, \qquad
    \Db W\xrightarrow{\quad \text{VdB}\quad} \Db \Lambda_{N-1}, \qquad
    \Db\Lambda_i \xleftrightarrows[\quad\Phi_i\quad]{\Phi_i} \Db \Lambda_{i+1}
    \quad (i\in \bbZ/N\bbZ).
\end{equation}
The number \(N\)(\(\geq 1\)) of algebras involved is called the \emph{helix period}
associated to \(\pi\), and depends only on the \emph{length} \(\ell\) of the
exceptional fiber \cite[see][\S 2.1]{donovanStringyKahlerModuli}. The
\emph{mutation functors} \(\Phi_i,\Phi_i\) are induced by (non-commutative
deformations of) spherical objects, and provide a standard set of equivalences up to which we classify.

\begin{keytheorem}[{(= \ref{cor:twotermclassification},
\ref{lem:inductioncontinues}, \ref{lem:collapse})}]
    \label{keythm:objclassification}
    If a complex \(x\in \Db X\) is supported on \(C\) and satisfies
    \(\Hom^{<0}(x,x)=0\), then up to a shift and iterated applications of the
    functors appearing in \eqref{eqn:allfunctors} and their inverses, \(x\) is
    either
    \begin{enumerate}[(1)]
        \item[{\crtcrossreflabel{(c1)}[item:objclass1]}]
            a \(\Lambda_i\)-module for some \(i\in
            \bbZ/N\bbZ\) , or
        \item[{\crtcrossreflabel{(c2)}[item:objclass2]}]
            a direct sum of shifts of coherent sheaves with
            zero--dimensional support on \(X\) (or on \(W\)), or
        \item[{\crtcrossreflabel{(c3)}[item:objclass3]}]
            a coherent sheaf on \(X\) (or on \(W\)), or a
            two--term complex thereof determined as a Yoneda extension \(x\in
            \Ext^2(\mathcal{G},\mathcal{F})\) between coherent sheaves
            \(\mathcal{F},\mathcal{G}\) such that
            \(\Hom(\mathcal{G},\mathcal{F})=0\) and \(\dim(\Supp\mathcal{G})=0\).
    \end{enumerate}
\end{keytheorem}

Much like Hara--Wemyss' treatment \cite{haraSphericalObjectsDimensions} of
analogous problems on the subcategory \(\ker(\mathbf{R}\pi_\ast)\subset \Db X\),
we prove \cref{keythm:objclassification} by attempting to inductively improve
the `spread' of a complex until it lies in a known heart.\\
\begin{minipage}[c]{0.57\textwidth}
    This begins by using its cohomologies with respect to
    \(H=\rflmod\Lambda_0\) to construct an intermediate heart \(K\), such that
    the complex lies in strictly fewer degrees with respect to \(K\) than with
    respect to \(H\) (\cref{prop:tiltingimproves}). These are dimension--agnostic
    arguments that apply equally well to all settings --- in particular, the
    same proof that applies to smooth surfaces also works for (possibly
    singular) 3-folds.
\end{minipage}\hspace{0.03\textwidth}
\begin{minipage}[c]{0.4\textwidth}
        \includegraphics[width=\textwidth]{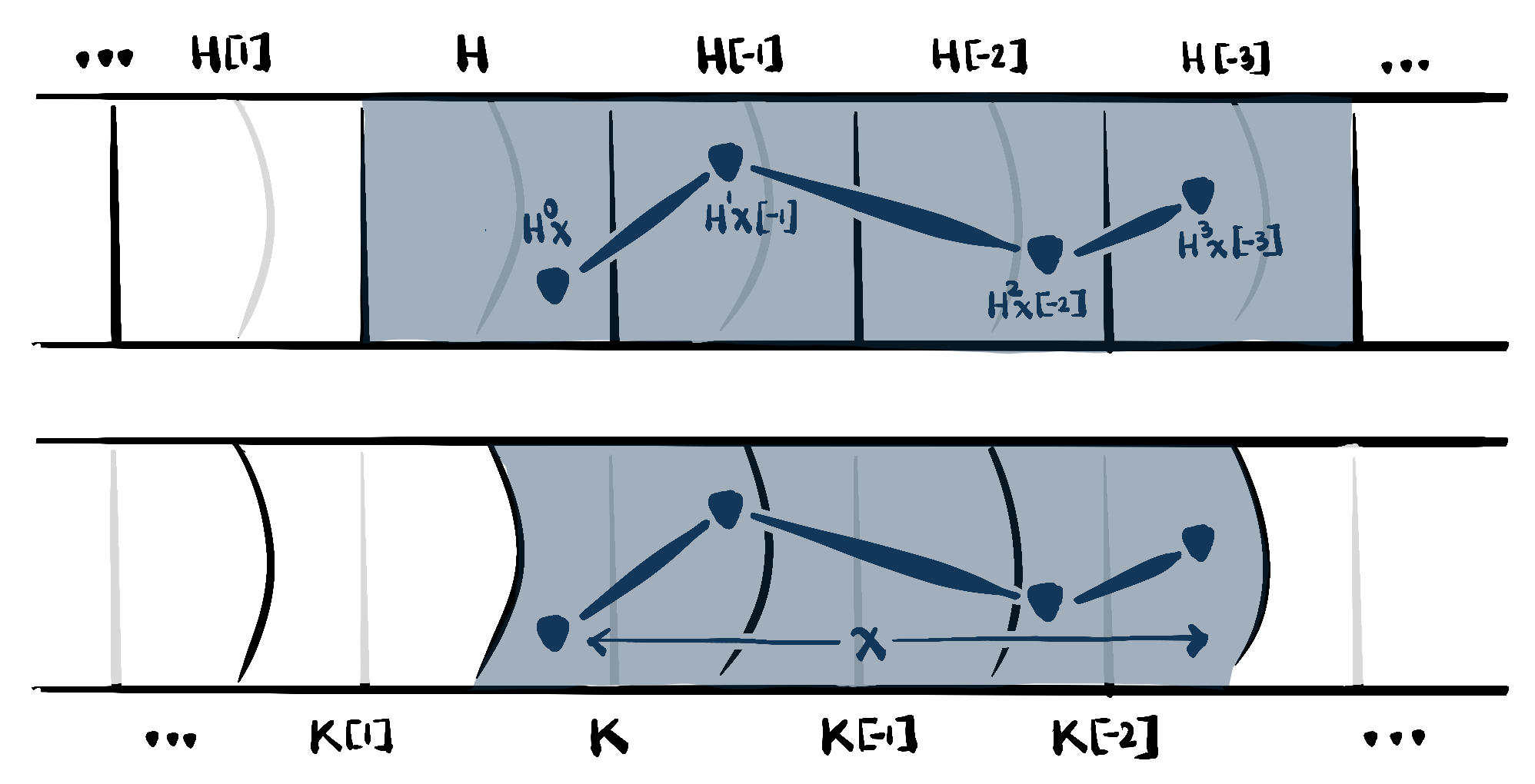}
\end{minipage}

Much unlike Hara--Wemyss' category however, mutation alone may not suffice to
reach the heart \(K\). Indeed there are complexes (such as
\(\OO_p\oplus\OO_q[1]\) for \(p,q\in X\)) which can only be improved by
tilting to a non-algebraic heart, and prima facie it is unclear how one
continues the induction from here. By considering the complete lattice of
\emph{all} intermediate hearts \cite[following][]{shimpiTorsionPairs3fold} and exploiting the irreducibility of \(C\), we obtain
precise cohomological information about objects which do not improve by mutation
(\cref{lem:collapse}), thus proving \cref{keythm:objclassification} in \cref{sec:thmAproof}.

\subsection*{Bricks and spheres} A classification of spherical objects, and
more generally semibricks (i.e.\ complexes whose self-extensions resemble
those of simple modules) then follows from \cref{keythm:objclassification}, the
classification of semibrick \(\Lambda_0\)--modules \cite[theorem
B]{shimpiTorsionPairs3fold}, and Donovan--Wemyss' analysis of simple
\(\Lambda_i\)-modules \cite[\S 1.2]{donovanStringyKahlerModuli}.

\begin{keytheorem}[(= \ref{cor:sbrickclassification})]
    \label{keycor:brickclassification}
    If a complex \(x\in \Db X\) supported on \(C\) satisfies
    \(\Hom^{<0}(x,x)=0\) and \(\Hom(x,x)=\bbC\), then up shifts and iterated
    applications of the functors appearing in \eqref{eqn:allfunctors} and their
    inverses, \(x\) is either
    \begin{enumerate}[(1)]
        \item[{\crtcrossreflabel{(b1)}[item:brickclass1]}] a simple
            \(\Lambda_i\)-module for some \(i\in \bbZ/N\bbZ\), or
        \item[{\crtcrossreflabel{(b2)}[item:brickclass2]}] a skyscraper sheaf at
            a closed point on \(X\), or on \(W\).
    \end{enumerate}
    In fact if \(x\) is described by \ref{item:brickclass1}, then up to shifts,
    mutation functors, and the Grothendieck duality functor, it is one of the
    sheaves \(\OO_C,\OO_{2C},...,\OO_{\ell C}\), or the unique
    non-split extension \(\mathcal{Z}\) of \(\OO_{2C}\) by \(\OO_{3C}\) which
    exists when \(\ell \geq 5\).
\end{keytheorem}

Evidently skyscraper sheaves at closed points are not spherical, i.e.\ any
spherical object must arise from a simple \(\Lambda_i\)-module as in the above
result. A direct computation \cite[corollary 7.5]{donovanStringyKahlerModuli}
reveals that when \(X\) is a smooth 3-fold, a spherical object can only arise
(as in \cref{keycor:brickclassification}) from the sheaves \(\OO_{\ell C}\) or
\(\mathcal{Z}\), and whether these are spherical depends on the specific
geometry. For instance \(\mathcal{Z}\) cannot be spherical unless \(\ell=5\),
and \(\OO_{\ell C}\) is spherical if and only if the associated top
Gopakumar--Vafa invariant is \(1\). In the special case when \(X\) is the
conifold resolution (i.e.\ when \(X\dashrightarrow W\) is the Atiyah flop), this
provides a purely algebro-geometric proof of the fact that every spherical
object is in the orbit of \(\OO_C\) under standard autoequivalences --- a statement
that was first proved by Keating--Smith \cite[theorem
1.3]{keatingSymplectomorphismsSphericalObjects2024} via symplectic dynamics on
the mirror manifold.

\subsection*{Tilting complexes, algebraic t-structures, and stability
conditions} Bricks detect Morita equivalences such as \eqref{eqn:allfunctors} by
picking out images of simple modules; traditionally such concerns
are addressed via projectives instead by considering \emph{tilting objects} ---
these are complexes \(t\in \Db X\) which satisfy \(\Hom^i(t,t)=0\) whenever
\(i\neq 0\) and generate the subcategory of perfects, thus inducing a derived equivalence between \(X\) and \(\End(t)\) \cite{rickardMoritaTheoryDerived}.

Indeed Van den Bergh \cite{vandenberghThreedimensionalFlopsNoncommutative}
constructs the equivalence \(\text{VdB}:\Db X\to \Db \Lambda_0\) by identifying
an indecomposable vector bundle \(\mathscr{N}\in \Coh X\) with globally
generated dual, such that the object \(\mathscr{N}\oplus \OO_X =\text{VdB}^{-1}(\Lambda_0)\) is tilting. Replacing either summand with its left or right
approximation by the other (i.e. \emph{mutating}) constructs four more tilting objects, and hence four equivalences out of \(\Db \Lambda_0\) --- these are the functors \(\Phi_0^\pm,\Phi_{-1}^\pm\) in \eqref{eqn:allfunctors}. The remaining
equivalences are likewise obtained by systematically propagating the mutation,
thus enumerating a class of tilting objects that are \emph{basic} (i.e.\ have
non--isomorphic indecomposable summands) and sit in a tetravalent
exchange graph. We show that there are no more such complexes.

\begin{keytheorem}[(= \ref{cor:tiltingclassification})]
    \label{keythm:tiltingclassification}
    Every basic tilting complex in \(\Db X\) can be obtained from a shift of Van den
    Bergh's tilting bundle \(\mathscr{N}\oplus \OO_X\) by finitely many left or
    right mutations in indecomposable summands. Equivalently, every basic tilting
    complex of \(\Lambda_i\)--modules (\(i\in \bbZ/N\bbZ\)) can be obtained by
    iteratively mutating indecomposable summands of \(\Lambda_i[n]\) for some
    \(n\in \bbZ\).
\end{keytheorem}

The classical tilting theory toolkit excels at populating a list of tilting
objects out of a single such instance, it however generally falls short of
identifying when such a list is exhaustive (notable exceptions are the cases of
Dynkin preprojective algebras and Donovan--Wemyss' contraction algebras, where
Aihara--Mizuno \cite{aiharaClassifyingTiltingComplexes} and August
\cite{augustFinitenessDerivedEquivalence} leverage the finiteness of involved
\(g\)--fans to obtain a complete classification.)
\Cref{keythm:tiltingclassification} exhibits the first class of \(g\)--tame (in
particular, silting--indiscrete) algebras for which the issue can be settled,
namely two--vertex contracted affine preprojective algebras and certain
modification algebras (algebras in these classes arise as \(\Lambda_i\)s when \(\dim Z=2\)
or \(3\) respectively.)

The proof is a \emph{consequence} of a classification of Morita equivalences,
more precisely we identify collections that could be the images of simple
modules under the inclusion of a bounded heart \(\rmod\Lambda\into \Db
\Lambda\isoto\Db X\) for some module--finite \(R\)--algebra \(\Lambda\). Any
such \(R\)--linear equivalence (including the functors in
\eqref{eqn:allfunctors}) restricts to the full subcategories supported on the
singular point \(\mathbf{0}\in Z\), i.e.\ the subcategories
\[\Dm X = \{x\in \Db X\;|\; \mathop\text{Supp}(x)\subseteq C\},\qquad
\Dfl\Lambda=\{x\in \Db \Lambda\;|\; \text{each }\HH^i(x)\in \rmod\Lambda\text{
has finite length}\}.\]
The natural heart \(\rflmod\Lambda=\Dfl\Lambda\cap\rmod\Lambda\) is the
extension--closure of its finitely many simples, whose images in \(\Dm X\) form
a \emph{simple--minded collection} which can be tracked and identified using
\cref{keythm:objclassification}.

\begin{keytheorem}[(= \ref{cor:sbrickclassification})]
    \label{keythm:heartclassification}
    Given any bounded t-structure with algebraic (i.e.\ Artinian and Noetherian) heart \(H\subseteq\Dm X\), there is an
    index \(i\in \bbZ/N\bbZ\) and equivalence \(\Dm X\to \Dfl\Lambda_i\)
    composed of functors appearing in \eqref{eqn:allfunctors} and their inverses
    that restricts to an equivalence between \(H\) and a shift of the full
    subcategory \(\rflmod\Lambda_i\) of finite length \(\Lambda_i\)-modules.
\end{keytheorem}

Another key consequence of \cref{keythm:heartclassification} is the
connectedness of the manifold parametrising Bridgeland stability conditions on
\(\Dm X\), an object of interest not least because it allows for an analysis of
the autoequivalence group of \(\Db X\). Such analyses were carried out by
Bridgeland \cite{bridgelandStabilityConditionsKleinian} and Hirano--Wemyss
\cite{hiranoStabilityConditions3fold}, where they found that all stability
conditions arising from the standard hearts \(\rflmod\Lambda_i\subset \Dfl
\Lambda_i\) (\(i\in \bbZ/N\bbZ\)) lie in a single connected component which is
preserved by autoequivalences composed of the mutation functors
\eqref{eqn:allfunctors}.

Then \cref{keythm:heartclassification} implies that a stability condition
outside this component, if it exists, must only correspond to t-structures with
non--algebraic hearts, in particular admitting no gaps in its phase distribution. By using \cref{keycor:brickclassification} to explicitly control
which phases can be occupied by any collection of stable objects, we preclude
the possibility of such a stability condition existing and deduce the following.

\begin{keytheorem}[(= \ref{cor:stabconnected})]
    \label{keythm:stabconnected}
    Given any locally finite Bridgeland stability condition \((Z,P)\) on \(\Dm
    X\), there is a phase \(\varphi\in \bbR\) such that the heart
    \(P\left(\varphi, \varphi+1\right]\) coincides with \(\Phi
    (\rflmod\Lambda_i)\), for some index \(i\in \bbZ/N\bbZ\)
    and an equivalence \(\Phi:\Db\Lambda_i\to \Db X\) composed of the functors
    in \eqref{eqn:allfunctors} and their inverses. In particular the stability
    manifold \(\StabX\) is connected.
\end{keytheorem}

\subsection*{Faithful monodromy on K\"ahler moduli} Viewing \(\Db X\) as
the topological B-model of a superconformal field theory
\cite[see][]{bridgelandSpacesStabilityConditions}, one expects to recover a class
of symmetries \(G\subseteq \Aut(\Db X)\) via variation of the complexified K\"ahler
structure \cite[\S 4]{aspinwallPointsPointView} on a conjectured moduli space
modelled as a submanifold of \(G\backslash\StabX/\bbC\).

This proposal was brought to fruition
\cite{bridgelandStabilityConditionsKleinian,todaStabilityConditionsCrepant,hiranoStabilityConditions3fold,donovanStringyKahlerModuli}
by taking \(G\) to be the subgroup of \(R-\)linear Fourier--Mukai
equivalences (that preserve a distinguished connected component of
\(\StabX\))\footnote{\label{fn:connected}This clause is now redundant in
light of \cref{keythm:stabconnected}.} and computing that (said connected component of)\footnotemark[1] the double quotient \(G\backslash\StabX/\bbC\) is an \((N+2)\)-punctured
sphere, a space reasonably declared to be the \emph{stringy K\"ahler moduli
space} \(\SKMS\).

The system of functors \eqref{eqn:allfunctors} can then be seen as a
representation of the fundamental groupoid \(\Pi_1(\SKMS)\) with \(N+2\) basepoints as in the figure below, with \(G\) arising from the fundamental group at \(\Db X\).  Faithfulness of this representation is equivalent to simply--connectedness of \(\StabX\), and is hence intertwined with long--standing
questions on contractibility of stability manifolds.

\begin{figure}[H]
    \begin{minipage}[c]{0.4\textwidth}
    \caption{
        Labelling basepoints (marked {\color{red}\xmark}) with categories, and
        homotopy classes of paths with functors as shown gives a representation
        of the fundamental groupoid \(\Pi_1(\SKMS, N+1)\).\newline\newline Loops
        at a fixed basepoint then give autoequivalences of the label, for
        instance the anti-clockwise monodromy of \(\Db X\) around the north pole
        is given by the functor \(\text{VdB}^{-1}\circ \Phi_{-1}\circ ... \circ
        \Phi_0\circ \text{VdB}\), computed to coincide with the tensor--action
        of a generator of \(\Pic X\).
        Likewise monodromy around the south pole recovers the action of \(\Pic
        W\), and monodromies around the equatorial punctures correspond to twist
        functors in bricks described in
        \protect{\cref{keycor:brickclassification} \ref{item:brickclass1}
        \cite[\S 6.3]{donovanStringyKahlerModuli}}.
        }\label{fig:skms} \end{minipage}\hspace{-2em}
    \begin{minipage}[c]{0.7\textwidth}
        \begin{annotationimage}{width=\textwidth}{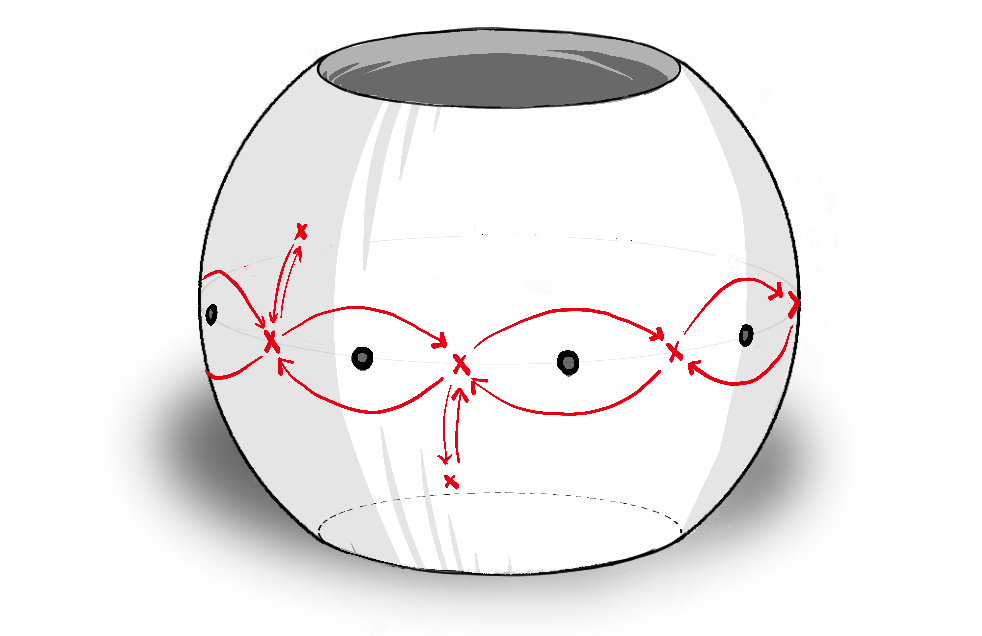}
            \draw[coordinate label = {{\tiny\(\Db\!\!\Lambda_0\)} at (0.51,0.43)}];
            \draw[coordinate label = {{\tiny\(\Db\!\!\Lambda_1\)} at (0.725,0.45)}];
            \draw[coordinate label = {{\tiny\(\Db\!\!\Lambda_2\)} at (0.845,0.52)}];
            \draw[coordinate label = {{\tiny\(\Db\!\!\Lambda_{\text{--}1}\)} at (0.32,0.46)}];
            \draw[coordinate label = {{\tiny\(\Db\! X\)} at (0.49,0.25)}];
            \draw[coordinate label = {{\tiny\(\Db\! W\)} at (0.34,0.64)}];
            \draw[coordinate label = {{\tiny\(\Phi_0\)} at (0.59,0.535)}];
            \draw[coordinate label = {{\tiny\(\Phi_0\)} at (0.59,0.315)}];
            \draw[coordinate label = {{\tiny\(\Phi_{\text{--}1}\)} at (0.37,0.54)}];
            \draw[coordinate label = {{\tiny\(\Phi_{\text{--}1}\)} at (0.37,0.325)}];
            \draw[coordinate label = {{\tiny\(\Phi_{\text{--}2}\)} at (0.195,0.39)}];
            \draw[coordinate label = {{\tiny\(\Phi_{1}\)} at (0.76,0.58)}];
            \draw[coordinate label = {{\tiny\(\Phi_{1}\)} at (0.76,0.37)}];
            \draw[coordinate label = {{\tiny\(\text{VdB}\)} at (0.268,0.57)}];
            \draw[coordinate label = {{\tiny\(\text{VdB}\)} at (0.485,0.32)}];
        \end{annotationimage}
    \end{minipage}
\end{figure}

Both questions --- of faithfulness and of contractibility --- can then be settled using Roberts--Woolf's upcoming work \cite{robertsContractibleStability}, where
they claim that each connected component of \(\Stab(\mathscr{T})\) admits a
contracting flow whenever \(\mathscr{T}\) is a triangulated category with
Grothendieck group \(\bbZ^{\oplus 2}\). \Cref{keythm:stabconnected},
combined with Roberts--Woolf's claim, thus proves the folklore conjecture recorded below for convenience.

\begin{keyconjecture}
    The stability manifold of \(\Dm X\) is homotopy equivalent to a point, and
    the quotient \(\StabX/\bbC\) is the universal cover of \(\SKMS\).
    Consequently the group \(G\) of \(R\)--linear Fourier--Mukai
    autoequivalences of \(\Db X\) is isomorphic to the fundamental
    group of \(\SKMS\), i.e.\ a free
    group on \(N+1\) generators.
\end{keyconjecture}

\subsection*{Acknowledgements} Many thanks to Michael Wemyss and Wahei Hara for
generously sharing their insights which guided this project; to Jon Woolf and
Stefan Roberts for explaining their work on stability spaces; and to Nick Rekuski and Timothy De Deyn for helpful comments.

\subsection*{Funding} The author was supported by ERC Consolidator Grant 101001227 (MMiMMa).

\subsection*{Open Access} For the purpose of open access, the author has applied a Creative Commons Attribution (CC:BY) licence to any Author Accepted Manuscript version arising from this submission.

\section{The taxonomy of objects}\label{sec:thmAproof}
Let \(\pi:X\to Z\) be the crepant contraction of an irreducible curve \(C\subset X\) as in the introduction,
and fix a complex \(x\in \Db X\) that is supported on \(C\) and satisfies
\(\Hom^i(x,x)\coloneq \Hom(x,x[i])=0\) whenever \(i<0\).

To prove \cref{keythm:objclassification}, we pass across Van den Bergh's
equivalence \(\text{VdB}=\RHom(\mathscr{N}\oplus \OO_X,-)\) \cite[see e.g.][\S 2.2]{donovanStringyKahlerModuli}. It is convenient to suppress \(\text{VdB}\) from the notation, making the identifications \(\Db X \simeq \Db \Lambda_0\) and \(\Dm X\simeq
\Dfl\Lambda_0\) and treating \(x\) as an object of either category. Thus working
with respect to the standard heart \(H=\rflmod\Lambda_0\) in \(\Dfl\Lambda_0\),
each cohomology object \(\HH^i(x)\in H\) is a finite--length
\(\Lambda_0\)--module, and \(x\) is an extension of the objects
\(\big\{\HH^i(x)[-i]\;|\; i\in \bbZ\big\}\). In particular, if \(a\leq b\) are
integers such that \(\HH^i(x)=0\) whenever \(i\notin [-b,-a]\), then \(x\) lies
in the full subcategory \(H[a,b]=\big\langle H[b] \cup H[b-1] \cup ... \cup
H[a] \big\rangle\), where \(\langle - \rangle\) denotes the
extension--closure.

\begin{notation}
    We write \(x\in H\llbracket a,b]\) to mean \(x\in H[a,b]\) and \(x\notin
    H[a+1,b]\) (equivalently, \(x\in H[a,b]\) and \(\HH^{-a}(x)\neq 0\)). Other
    variations such as \(x\in H[a,b\rrbracket\) and \(x\in H\llbracket
    a,b\rrbracket\) are defined likewise. The notation naturally extends when
    measuring cohomologies with respect to other bounded hearts \(K\subset \Dfl
    \Lambda_0\).
\end{notation}

Now we may, up to replacing \(x\) with a shift thereof, assume that \(x\) lies
in \(H\llbracket 0 , n\rrbracket\) for some integer \(n\geq 0\) (called the
\emph{spread} of \(x\).) If the equality \(n=0\) held true then \(x\in H\) would
be a finite length \(\Lambda_0\)--module as described in
\cref{keythm:objclassification} \ref{item:objclass1}, so it suffices to only consider
the case \(n>0\) and attempt to reduce the spread by updating \(H\). Our primary tool to accomplish this is \emph{tilting in
torsion pairs} \`a la Happel--Reiten--Smal\o\ \cite{happelTiltingAbelianCategories}.

\begin{lemma}
    \label{lem:spreadanalysis}
    Suppose \(T\subseteq H\) is a torsion class, corresponding to the torsion-free
    class \(F\subseteq H\) and tilted heart \(K=\langle F[1]\cup T\rangle\).
    \begin{enumerate}[(i)]
        \item The torsion class \(T\) contains the top cohomology \(\HH^0(x)\)
            if and only if \(x\) lies in \(K\llbracket 0, n]\). \label{item:spread1}
        \item The torsion-free class \(F\) contains the bottom cohomology
            \(\HH^{-n}(x)\) if and only if \(x\) lies in \(K[-1,n-1\rrbracket\)
            \label{item:spread2}
    \end{enumerate}

    \begin{proof}
        By construction, \(K\) lies in \(H[0,1]\) and hence \(H\) (in
        particular, each cohomology \(\HH^{-i}(x)\)) lies in \(K[-1,0]\). Now \(x\)
        is an extension of the objects \(\HH^{-i}(x)[i]\in K[i-1,i]\) for \(i\in
        [0,n]\), of which all except \(\HH^0(x)\) clearly lie in \(K[0,n]\).
        Thus if \(\HH^0(x)\) lies in \(T\subseteq K\), then \(\HH^0(x)\) (and hence
        also \(x\) itself) lies in \(K[0,n]\). It is also clear that in this
        case \(x\) cannot lie in \(K[1,n]\subset H[1,n+1]\), since
        \(\HH^0(x)\neq 0\) by assumption. Hence \(x\) lies in \(K\llbracket 0,
        n]\).

        Conversely if \(x\) lies in \(K\llbracket 0, n]\), then \(\Hom(x,k)=0\)
        for all \(k\in K[<0]\) from whence it follows (e.g.\ by truncating to
        the coaisle of \(H\)) that \(\Hom(\HH^0x,
        k)=0\) for all \(k\in K[<0]\). Thus \(\HH^0(x)\) lies in \(K[\geq 0]\),
        and hence in \(T=K[\geq 0]\cap H\). This proves \ref{item:spread1}; the
        proof of \ref{item:spread2} is similar.
    \end{proof}
\end{lemma}

The Abelian category \(H\) is Artinian and Noetherian, so the inclusion order on torsion--free (resp.\ torsion) classes in \(H\) forms a \emph{complete
lattice} \(\torf(H)\) (resp.\ \(\tors(H)\)) which is (anti--)isomorphic to the
poset of tilted hearts
\[
    \big[H,H[1]\big]
    =
    \big\{K\text{ heart of a bounded t-structure}\;|\; H\leq K \leq H[1]\big\},
\]
where the latter is considered an interval in the set of all t-structures
partially ordered by coaisle--containment \cite[see e.g.][\S
2.2]{shimpiTorsionPairs3fold}. That is to say, we have naturally isomorphic
partial orders \(\tors(H)^\text{op}\simeq \big[H,H[1]\big]\simeq
\torf(H)\) that admit arbitrary suprema and infima. In particular there is a minimal
torsion class containing \(\HH^0(x)\) with corresponding tilted heart
\(\ol{K}\), and a minimal torsion--free class containing \(\HH^{-n}(x)\) with
corresponding tilted heart \(\ul{K}\). \Cref{lem:spreadanalysis} then
characterises these hearts in terms of the partial order as
\[
    \big[H, \ol{K}\,\big]
    = \left\{K\in \big[H,H[1]\big]\;\middle\vert\; x\in K\llbracket 0, n]\right\},
    \qquad
    \big[\ul{K}, H[1]\big]
    = \left\{K\in \big[H,H[1]\big]\;\middle\vert\; x\in K[-1, n-1\rrbracket\right\}.
\]

\begin{lemma}
    \label{prop:tiltingimproves}
    We have the inequality \(\ul K \leq \ol K\). In particular, the poset of
    t-structures on \(\Dm X\) contains a \emph{non-empty interval} \(\big[\ul K, \ol
    K\,\big]\) containing hearts \(K\) such that \(x\) lies in \(K\llbracket 0,
    n-1\rrbracket\).

    \begin{proof}
        The statement claims the containment of torsion classes \(\ol K\cap H
        \subseteq \ul K \cap H\), hence it suffices to show that \(\HH^0(x)\) is
        contained in \(\ul K\cap H=\orth(\HH^{-n}x)\). In other words, we must
        show \(\Hom(\HH^0x, \HH^{-n}x)=0\). But this readily follows from the
        hypothesis \(\Hom(x, x[-n])=0\), by truncating first to the aisle and
        then to the coaisle and noting that the truncation functors are adjoint
        to inclusions (alternatively, by using the spectral sequence
        \(E_2^{p,q}=\bigoplus_i \Hom^p(\HH^i x, \HH^{i+q}x)\Longrightarrow
        \Hom^{p+q}(x,x)\) to deduce that \(\Hom(x,x[-n])=\Hom(\HH^0x, \HH^{-n}x)\).)
    \end{proof}
\end{lemma}

\subsection{Classifying two--term complexes} \Cref{prop:tiltingimproves} asserts
that the spread of \(x\) can \emph{always} be improved by tilting once, in
particular if \(x\) has spread \(n=1\) then it is an object in some tilted heart
\(K\in \big[H, H[1]\big]\). Such hearts are well--understood in terms of the
functors \eqref{eqn:allfunctors}, we briefly discuss the constructions
and classification.

The mutation functors can be defined from a \(\bbZ\)--indexed sequence of rigid
reflexive \(R\)--modules \(\{...,V_{-1},V_0,V_1,...\}\) ordered such that
\(V_{-1}\), \(V_0\) are pushforwards of \(\mathscr{N}\), \(\OO_X\) respectively, and any
three consecutive modules sit in an exchange sequence
\cite[see][\S 2.3]{donovanStringyKahlerModuli}. Writing \(M_i=V_{i-1}\oplus
V_{i}\), the modules \(\Hom_R(M_i,M_{i+1})\in \rmod\,(\End_RM_i)\) and
\(\Hom_R(M_{i+1},M_i)\in \rmod\,(\End_RM_{i+1})\) are both tilting and induce
\(R\)--linear equivalences
\[
    \Db (\End_RM_i)
    \xleftrightarrows[\quad\Psi_{i}\coloneq\RHom(\Hom_R(M_{i+1},M_i),-)\quad]{\quad\Psi_i\coloneq\RHom(\Hom_R(M_{i},M_{i+1}),-)\quad}
    \Db(\End_R M_{i+1})
\]
indexed over the mutated summand as in \cite[\S
3.2]{donovanStringyKahlerModuli}. Proposition 4.3 \emph{ibid} then exhibits the
inequalities of hearts \(\Psi_i\rflmod(\End_R M_{i+1}) < \rflmod(\End_R M_i) <
\Psi_i^{-1}\rflmod(\End_R M_{i+1})\) and further shows these relations are
\emph{covering} (i.e.\ the corresponding open intervals are empty.)

To recover the picture \eqref{eqn:allfunctors}
we note that the class group \(\mathop\text{Cl}(R)\cong \bbZ\) acts on the
sequence \((V_i)\), translating it in multiples of the helix period \(N\). In
particular we have canonical \(R\)--linear identifications
\(\beta_k:\End_RM_i\isoto \End_RM_{i+kN}\) for all \(i,k\in \bbZ\) (see \S 3.3
\emph{ibid}), so we declare \(\Lambda_i=\End_R(M_i)\) and \(\Phi_i=\Psi_i\) for
\(i\in \{0,...,N-1\}\) and observe that for any \(k\in \bbZ\) there is an equality
of functors \(\Psi_{i+kN}= \beta_{-k}\circ \Phi_i \circ \beta_k\)
\cite[lemma 7.3]{hiranoStabilityConditions3fold}. Reading off t-structures in
\(\Dfl\Lambda_0\) then recovers a part of the Hasse quiver of
\(\big[H,H[1]\big]\), given by
\begin{equation}
    \tag{!}
    \label{eqn:Hassequiv}
\begin{tikzcd}[row sep=tiny, column sep=tiny]
      &\Phi_0(\rflmod\Lambda_1)[1]\rar
      &\Phi_0\Phi_1(\rflmod\Lambda_2)[1]\rar
      &\cdots \rar
      &\Phi_{\text{-}1}^{\text{-}1}\Phi_{\text{-}2}^{\text{-}1}(\rflmod\Lambda_{\,\text{-}2})
       \rar
      &\Phi_{\text{-}1}^{\text{-}1}(\rflmod\Lambda_{\,\text{-}1})
       \arrow[dr, bend left=15]
      & \\
    H[1] \arrow[ur,bend left=15]\arrow[dr,bend right=15]
      &&&&&& H \\
      &\Phi_{\text{-}1}(\rflmod\Lambda_{\,\text{-}1})[1]\rar
      &\Phi_{\text{-}1}\Phi_{\text{-}2}(\rflmod\Lambda_{\,\text{-}2})[1]\rar
      &\cdots \rar
      &\Phi_0^{\text{-}1}\Phi_1^{\text{-}1}(\rflmod\Lambda_2)\rar
      &\Phi_0^{\text{-}1}(\rflmod\Lambda_1)\arrow[ur, bend right=15]
      &
\end{tikzcd}
\end{equation}
where the indices are modulo \(N\). In \cite[theorem
5.13]{shimpiTorsionPairs3fold}, numerical techniques are used to show that the
two decreasing chains starting at \(H[1]\) above have infima \(\text{VdB}(\coh
X)[1]\) and \(\Phi_0\circ \text{VdB}(\coh W)[1]\) respectively, so these geometric
hearts too lie in \(\big[H, H[1]\big]\). Further tilting the geometric hearts in
torsion classes containing only skyscrapers and extensions thereof gives a class
of `mixed geometric hearts' (see \S 5.3 \emph{ibid}) which are also tilts of
\(H\). The key result there is that this list is exhaustive and \(H\) has no
other tilts, with a short proof (pertinent to the single--curve case we are in)
sketched as \cite[Algorithm 1.1]{shimpiTorsionPairs3fold}.

\begin{theorem}[{\cite[Theorem A]{shimpiTorsionPairs3fold}}]
    \label{thm:tiltclassification}
    Let \(K\in \big[H ,H[1]\big]\) be the heart of a bounded t-structure in
    \(\Dfl\Lambda_0\), obtained by tilting \(H=\rflmod\Lambda_0\).
    Then \(K\) admits one of the following descriptions.
    \begin{enumerate}[(1)]
        \item[{\crtcrossreflabel{(h1)}[item:heartclass1]}]
            The heart \(K\) is the image of some standard
            heart \(\rflmod\Lambda_i\subset
            \Dfl\Lambda_i\) (\(i=\bbZ/N\bbZ\)) under a
            functor \(\Dfl\Lambda_i\to \Dfl \Lambda_0\) given by one of
            the composites
            \([1]\circ(\Phi_{mN-1}\cdots\Phi_{i+1}\Phi_i)\),
            \([1]\circ(\Phi_{-mN}\cdots\Phi_{i-2} \Phi_{i-1})\),
            \((\Phi_{mN+i-1}\cdots\Phi_1\Phi_0)^{-1}\), or
            \((\Phi_{i-mN}\cdots\Phi_{-2}\Phi_{-1})^{-1}\)
            (for some \(m\in \bbN\)).
        \item[{\crtcrossreflabel{(h2)}[item:heartclass2]}]
            Identifying \(\Dfl \Lambda_0\simeq \Dm X\) across the
            equivalence \(\textnormal{VdB}\), the heart \(K\) is a tilt of the
            standard heart \(\coh X= \Coh X \cap\Dm X\) in a (possibly
            trivial) torsion class containing no sheaves with one--dimensional
            support. In other words, there is a subset \(U\subseteq X\) such that
            \(K=\left\langle\{\OO_p\;|\;p\in U\}\cup
                \{\mathcal{F}[1]\;|\;\mathcal{F}\in \coh X,\; \Hom(\OO_p,
            \mathcal{F})=0\text{ for all }p\in U\}\right\rangle\).
        \item[{\crtcrossreflabel{(h3)}[item:heartclass3]}]
            Identifying \(\Dfl \Lambda_0\simeq \Dm W\) across the
            equivalence \(\Phi_0\circ\textnormal{VdB}\), the heart \(K\subset \Dm
            W\) is a tilt of the standard heart \(\coh W\) in
            a (possibly trivial) torsion class containing no sheaves with
            one--dimensional support. In other words, there is a subset
            \(U\subseteq W\) such that
            \(K=\left\langle\{\OO_p\;|\;p\in U\}\cup
                \{\mathcal{F}[1]\;|\;\mathcal{F}\in \coh W,\; \Hom(\OO_p,
            \mathcal{F})=0\text{ for all }p\in U\}\right\rangle\).
    \end{enumerate}
\end{theorem}

The following
consequence is immediate.

\begin{corollary}
    \label{cor:twotermclassification}
    If the complex \(x\in \Dfl\Lambda_0\) lies in \(H[0,1]\), then \(x\) adheres to one of the
    descriptions in \cref{keythm:objclassification}.

    \begin{proof}
        \Cref{prop:tiltingimproves} guarantees that \(x\)
        lies in some heart \(K\) obtained as a tilt of
        \(H=\rflmod\Lambda_0\). If \(K\) admits the description
        \ref{item:heartclass1} in \cref{thm:tiltclassification}, then
        \(K=\Phi\rflmod\Lambda_i\) for some \(i\in \bbZ/N\bbZ\) and \(\Phi\) a
        composite of shifts and mutation functors, hence \(\Phi^{-1}(x)\in
        \rflmod\Lambda_i\) is a \(\Lambda_i\)--module as in
        \cref{keythm:objclassification} \ref{item:objclass1}.

        Otherwise the heart \(K\) admits one of the descriptions
        \ref{item:heartclass2} or \ref{item:heartclass3}, say (without loss of
        generality) the former.  Thus \(K\) is a tilt of \(\coh X\) in a torsion
        class \(\langle \OO_p\;|\; p\in U\rangle\subset \coh X\), determined by
        some subset \(U\subset X\). It follows that the object \(x\in K\) can be
        written as an extension \(\mathcal{F}[1]\to x \to \mathcal{G}\to
        \mathcal{F}[2],\) where \(\mathcal{F}\) and \(\mathcal{G}\in \coh X\) are
        torsion--free (resp.\ torsion) with respect to the torsion pair relating
        \(K\) and \(\coh X\). Thus \(\mathcal{G}\in \langle \OO_p\;|\; p\in
        U\rangle\) has zero--dimensional support,
        \(\Hom(\mathcal{G},\mathcal{F})=0\), and \(x\) is determined by the
        morphism \(\mathcal{G}\to \mathcal{F}[2]\) seen as an element of
        \(\Ext^2(\mathcal{G},\mathcal{F})\).
    \end{proof}
\end{corollary}

\subsection{Induction, dream scenario} \label{subsec:dreaminduction} Given
\(x\in H\llbracket 0, n\rrbracket\) for \(n\geq 2\), we may hope
that the interval \(\big[\ul K, \ol K\,\big]\subset \big[H,H[1]\big]\) contains
a heart of the form \(\Phi H\) for some \emph{auto}equivalence \(\Phi:\Dfl
\Lambda_0\to \Dfl \Lambda_0\) of desired form, so that replacing \(x\) by
\(\Phi^{-1}(x)\in H\llbracket 0,n-1\rrbracket\) allows us to rinse and repeat
until \cref{cor:twotermclassification} kicks in.

And we are correct in hoping so --- the sub-poset of \(\big[H, H[1]\big]\)
containing only algebraic hearts is given by the Hasse quiver
\eqref{eqn:Hassequiv}, so unless the window \(\big[\ul K, \ol K\,\big]\) is
exceptionally narrow it is bound to contain an image of \(H=\rflmod\Lambda_0\).
In fact if we allow ourselves an additional tilt, then we can arrange for our
hopes to be true as long as \(\big[\ul K, \ol K\,\big]\) contains \emph{at least
one} algebraic heart, e.g.\ if either if the two end points are algebraic.

\begin{lemma}
   \label{lem:inductioncontinues}
   If either \(\ol K\) or \(\ul K\) is algebraic, i.e.\ fits the description
   \ref{item:heartclass1} in \cref{thm:tiltclassification}, then there is an
   autoequivalence \(\Phi:\Dfl \Lambda_0\to \Dfl \Lambda_0\) composed of shifts,
   mutation functors, and their inverses such that \(x\) lies in \(\Phi(H)[ 0,
   n-1]\).
\end{lemma}

Before moving on to the proof we sketch an illustrative albeit contrived
example. Writing \(s_0,s_1\in \rflmod\Lambda_1\) for the two simple modules
indexed as in \cite[\S 3.1]{donovanStringyKahlerModuli}, the
non--zero cohomologies of \(x=\Phi_0s_0[3]\oplus\Phi_0s_1[1] \in
\Dfl\Lambda_0\) with respect to \(H\) are given by \(\HH^0(x)=\Phi_0s_1[1]\)
and \(\HH^{-3}(x)=\Phi_0s_0\). In particular \(\HH^0(x)\) is simple in \(H\) and
\(\HH^{-3}(x)\) generates the torsion--free class \(\HH^0(x)\orth\), so the only
heart in \(\big[H,H[1]\big]\) which improves the spread of \(x\in H\llbracket
0,3\rrbracket\) is \(\ul K=\ol K =\Phi_0(\rflmod\Lambda_1)[1]\).
But looking at \(x'=\Phi_0^{\text{-}1}(x)[-1]\in H'\llbracket 0,2\rrbracket\), where
\(H'=\rflmod\Lambda_1\), it is easy to see that further tilting \(H'\) in any
torsion class containing \(s_1\) will not worsen the spread of \(x'\). Thus
\(x'\) has spread \(\leq 2\) with respect to each of the hearts
\(\Phi_0(\rflmod\Lambda_0)[1]\),
\(\Phi_0\Phi_{\text{-}1}(\rflmod\Lambda_{\text{-}1})\),
\(\Phi_0\Phi_{\text{-}1}\Phi_{\text{-}2}(\rflmod\Lambda_{\text{-}2})\), ... It
follows that each of the objects \((\Phi_0\Phi_{\text{-}1}...
\Phi_{N\text{-}1})^{-k}\circ \Phi_{0}^{\text{-}1}\Phi_{0}^{\text{-}1}x[-2]\)
(\(k\geq 0\)) lies in \(H[0,2]\).

\begin{proof}[Proof of \cref{lem:inductioncontinues}]
       The hypothesis guarantees that there is a heart
       \(\Psi(\rflmod\Lambda_i)\in \big[H,H[1]\big]\), where \(\Psi\) is
       composed of shifts and (inverse) mutation functors, such that \(x\) lies in
       \(\Psi(\rflmod\Lambda_i)\llbracket 0,n-1\rrbracket\). Equivalently,
       \(x'=\Psi^{-1}(x)\) lies in \(H'\llbracket 0, n-1\rrbracket\) where
       \(H'=\rflmod\Lambda_i\).

       Now the same argument as in \cref{lem:spreadanalysis} shows that the
       poset of tilted hearts \(\big[H',H'[1]\big]\) contains intervals
        \[
            \big[H', \ol{K}{}'\,\big]
            = \left\{K\in \big[H',H'[1]\big]\;\middle\vert\; x\in K\llbracket 0,
            n-1]\right\},
            \qquad
            \big[\ul{K}', H'[1]\big]
            = \left\{K\in \big[H',H'[1]\big]\;\middle\vert\; x\in K[-1,
            n-2\rrbracket\right\}
        \]
        defined using the top and bottom cohomologies of \(x'\) with respect to
        \(H'\). The vanishing of \(\Hom^{<0}(x',x')\) implies the
        inequality \(\ul K' \leq \ol K{}'\) as in \cref{prop:tiltingimproves}.

        This implies that the union \(\big[H', \ol K{}'\big]\cup \big[\ul K',
        H'[1]\big]\) contains an increasing chain from \(H'\), with
        each \(K\) in the chain satisfying \(x'\in K[m, m+n-1]\) for some
        \(m\in\{0,-1\}\). Since an algebraic heart cannot cover or be covered by a
        non-algebraic heart \cite[corollary 4.7]{shimpiTorsionPairs3fold} and
        the Hasse quiver of algebraic hearts in \(\big[H',H'[1]\big]\) is
        identical to \eqref{eqn:Hassequiv} with indices shifted up by
        \(i\), we see that the chain must contain some \(K\) of the form
        \(\Phi(\rflmod\Lambda_0)\) where \(\Phi\) is composed of mutation
        functors. Then \(\Phi^{-1}(x')[-m]=\Phi^{-1}\Psi^{-1}(x)[-m]\)
        lies in \((\rflmod\Lambda_0)[0,n-1]\) as required.
\end{proof}

\subsection{Proverbial spanners in the inductive works} We now address complexes
for which the induction outlined in \cref{subsec:dreaminduction} fails to
continue, i.e.\ complexes \(x\in H\llbracket 0,n\rrbracket\) (\(n\geq 2\)) for
which both \(\ul K\) and \(\ol K\) are non-algebraic and fit the descriptions
\ref{item:heartclass2} or \ref{item:heartclass3} in
\cref{thm:tiltclassification}.

Note that hearts described in \ref{item:heartclass2} are clearly incomparable
(under the coaisle--containment order) with those in \ref{item:heartclass3}, so
replacing \((x, \Lambda_0, X)\) with \((\Phi_{\text{-}1}^{\text{-}1}x,
\Lambda_{\text{-}1}, W)\) if
necessary we may assume the hearts \(\ul K, \ol K\) are both (without loss of
generality) tilts of \(\coh X\) in some torsion classes containing only sheaves
with zero--dimensional support.

We show in this case that \(x\) fits the description \ref{item:objclass2} in
\cref{keythm:objclassification}, i.e.\ decomposes into shifts of sheaves
contained in the extension--closure \(\langle \OO_p\;|\;p\in C\rangle\). It is
classical \cite[see e.g.][\href{https://stacks.math.columbia.edu/tag/00KY}{tag
00KY}]{stacks-project} that such sheaves (henceforth dubbed
\emph{zero--dimensional sheaves}) can be characterised by the property
that their support is zero--dimensional, equivalently the support (being closed
in an irreducible curve) has non-empty complement in \(C\).  Support can
detected by looking at morphisms from skyscrapers (and extensions thereof) as follows.

\begin{lemma}
    \label{lem:easyskyscraperdetection}
    Suppose non-zero sheaves \(u,y\in \coh X\) (supported within \(C\)) are such
    that \(u\in \langle \OO_p\;|\; p\in X\rangle\), and either \(\Hom(y,u)=0\) or
    \(\Hom(u,y)=\Ext^1(u,y)=0\) holds. Then \(\dim(\Supp y)=0\).

    \begin{proof}
        The conclusion is immediate from the first hypothesis since if \(y\) was
        supported everywhere on \(C\) then for any \(p\in \Supp(u)\) we would
        have a non-zero morphism \(y\onto \OO_p\into u\).

        To deduce this from the second hypothesis, suppose for the sake of
        contradiction that \(\dim(\Supp y)=1\) and
        \(\Hom(u,y)=\Ext^1(u,y)=0\). Thus in particular \(\sheafHom(u,y)=0\)
        where we note that the homomorphism sheaf is supported on a finite set
        of closed points \(\Supp(u)\) so vanishes if and only if its module of
        global sections does.

        Likewise we can deduce (e.g.\ from the Grothendieck spectral sequence associated to
        \(\HH^0(X,-) \circ \mathbf{R}\curlyHom(u,-)\)) that
        \(\HH^0(X,\sheafExt^1(u,y))=\Ext^1(u,y)=0\), and hence
        \(\sheafExt^1(u,y)=0\).

        Taking stalks at \(p\in \Supp(u)\) we thus compute the equality of
        \(\OO_{X,p}-\)modules \(\Hom(u_p,y_p)=\Ext^1(u_p,y_p)=0\), here the
        stalks \(u_p,y_p\) are finitely generated modules over the local ring
        \(\OO_{X,p}\) and \(u_p\) is supported on the maximal ideal (in fact is
        filtered by the residue field.) It follows
        that \(\text{depth}(y_p)\geq 2\) \cite[theorem
        16.6]{matsumuraCommutativeRingTheory} but this is absurd since the depth
        of a module is bounded above by the dimension of its support
        \cite[\href{https://stacks.math.columbia.edu/tag/00LK}{tag
        00LK}]{stacks-project}.
    \end{proof}
\end{lemma}

The result above can be extended to objects of the heart
\(H=\text{VdB}^{-1}(\rflmod\Lambda_0)\).

\begin{lemma}
    \label{lem:skyscraperdetection}
    An object \(y\in H\) lies in the extension closure \(\langle \OO_p\;|\; p\in
    X\rangle\) if any of the following conditions hold.
    \begin{enumerate}[(i)]
        \item \(\Hom(u,y)=\Ext^1(u,y)=0\) for some non-zero \(u\in \langle
            \OO_p\;|\; p\in C \rangle\).
            \label{item:skyscraper1}
        \item \(\Hom(y,u)=\Ext^1(y,u)=0\) for some non-zero \(u\in \langle
            \OO_p\;|\; p\in C \rangle\).
            \label{item:skyscraper3}
        \item \(\Hom(u,y)=\Hom(y,v)=0\) for some (not necessarily
            distinct) non-zero \(u,v\in \langle \OO_p\;|\; p\in C \rangle\).
            \label{item:skyscraper2}
    \end{enumerate}

    \begin{proof}
        Recall \cite[see e.g.][\S
        3.1]{vandenberghThreedimensionalFlopsNoncommutative} that \(H\) (when
        considered a subcategory of \(\Dm X\simeq \Dfl\Lambda_0\)) coincides
        with the tilt \({}^0\!\operatorname{Per}(X/Z)\cap \Dm X\) of \(\coh X\),
        hence the object \(y\in H\) sits in an exact triangle
        \[
            y'[1] \to y \to y'' \to y'[2]
        \]
        where \(y',y''\in \coh X\) satisfy
        \(\pi_\ast(y')=\mathbf{R}^1\pi_\ast(y'')=0\) and
        \(\Hom(\OO_C(-1),y')=0\). In particular \(y'\) has no subsheaves with
        zero--dimensional support, i.e.\ \(\Hom(u,y')=0\) for all \(u\in
        \langle\OO_p\;|\; p\in C\rangle\) .

        For a non-zero \(u\in \langle \OO_p\;|\; p\in C\rangle\), a straightforward long exact sequence argument shows there is an inclusion
        \(\Ext^1\!(u,y')\subseteq \Hom(u,y)\), thus if either
        \ref{item:skyscraper1} or \ref{item:skyscraper2} holds then the
        hypothesis \(\Hom(u,y)=0\) implies \(\Ext^1\!(u,y')=0\) which
        cannot happen unless \(\dim(\Supp y')<1\) by
        \cref{lem:easyskyscraperdetection}. This
        necessarily implies \(y'=0\), and hence \(y=y''\) is a sheaf. The same
        lemma then lets us deduce the condition on \(\Supp(y)\) from either pair
        of hypotheses.

        On the other hand if \ref{item:skyscraper3} holds, then the vanishing of
        \(\Hom(y'',u)=\Hom(y,u)\) shows \(\Supp(y'')\) is disjoint from \(\Supp(u)\). The long exact sequence \(...\to
        \Ext^1\!(y,u)\to\Hom(y',u)\to \Ext^2\!(y'',u)\to...\) then shows
        \(\Hom(y',u)\) vanishes too, which can only happen if \(y'=0\) and hence
        \(y=y''\) is of the claimed form.
    \end{proof}
\end{lemma}
\medskip

We are now sufficiently equipped to witness the decomposition of \(x\) --- the
constraints on \(\ol K\) show the cohomology object \(\HH^0(x)\in
H\) surjects onto a skyscraper sheaf, which we then exploit against \cref{lem:skyscraperdetection}.

\begin{lemma}
    \label{lem:collapse}
    Suppose \(x\in H\llbracket 0,n\rrbracket\) (\(n\geq 2\)) is such that
    the hearts \(\ul K,\ol K\) are both tilts of \(\coh X\) in (possibly
    trivial) torsion classes containing sheaves with zero--dimensional support.
    Then each cohomology object \(\HH^i(x)\in H\) lies in
    \(\left\langle \OO_p\;|\; p\in X\right\rangle\), and cohomologies in
    distinct degrees have disjoint supports.

    \begin{proof}[Proof that \(\HH^0(x)\) and \(\HH^{-n}(x)\) have
        zero--dimensional supports]
        The conditions on \(\ol K\) guarantee the containments of torsion classes
        \[
            \coh X[1] \cap H \quad \subseteq \quad \ol K \cap H \quad \subseteq
            \quad \langle \coh X[1]\cap H,\, \{\OO_p\;|\; p\in C\}\rangle ,
        \]
        and the first containment must be proper --- indeed, by
        \cite[lemma 5.16]{shimpiTorsionPairs3fold} the torsion class \(\coh
        X[1]\cap H\) is the supremum (i.e.\ union) of a strictly increasing
        chain \(0\subset T_1\subset T_2 \subset ...\) in
        \(\tors(H)\), so if \(\ol K \cap H\) is equal to this union then
        \(\HH^0(x)\in \ol K \cap H\) is contained in some \(T_i\).  But this is
        absurd, since \(\ol K\cap H\) is by definition the minimal torsion class
        containing \(\HH^0(x)\).

        Thus \(\HH^0(x)\) lies in \(\langle \coh X[1]\cap H,\, \{\OO_p\;|\; p\in
        C\}\rangle \setminus (\coh X[1]\cap H)\), sitting in a
        triangle
        \(y[1]\to \HH^0(x)\to z\to y[2]\)
        where \(y,z\in \coh X\) are sheaves such that \(y\in H[-1]\), and \(z\in
        \langle\OO_p\;|\;p\in C\rangle\) is non--zero.

        Now the analogous condition \(\ul K[-1] \cap H\subseteq \coh X \cap H\)
        implies that \(\HH^{-n}(x)\) is a sheaf, so \(\Hom^{<0}(y,\HH^{-n}x)\)
        vanishes and the above exact triangle gives the inclusions
        \(\Hom(z,\HH^{-n}x[i])\subseteq \Hom(\HH^0x,\HH^{-n}x[i])\) for
        \(i=0,1\) via a long exact sequence.

        But the conditions \(\Hom(x,x[-n])=\Hom(x,x[1-n])=0\), after truncating
        to the aisle and the coaisle of \(H\) as in the proof of
        \cref{prop:tiltingimproves}, imply
        \begin{equation}
            \Hom(\HH^0x,\HH^{-n}x)=\Hom{\!}^1(\HH^0x, \HH^{-n}x)=0.
            \label{eqn:homzero}\tag{!!}
        \end{equation}
        Thus we conclude \(\Hom(z,\HH^{-n}x)=\Hom^1(z,\HH^{-n}x)=0\), so that
        \(\HH^{-n}(x)\) is a sheaf with zero--dimensional support outside
        \(\Supp(z)\) by \cref{lem:skyscraperdetection}. It immediately follows
        (from \eqref{eqn:homzero} and \cref{lem:skyscraperdetection}) that
        \(\HH^0(x)\) is also a sheaf with zero--dimensional support, in
        particular \(\HH^0(x)=z\) and \(y=0\).
    \end{proof}

    \begin{proof}[Proof that the remaining cohomologies have zero--dimensional
        supports]
        One can continue to use truncation functors and analyse
        long exact sequences to compare the cohomologies \(\HH^i(x)\) against
        \(\HH^0(x)\) and \(\HH^{-n}(x)\), we find it slightly
        more convenient to instead examine the spectral sequence \cite[see
        e.g.][IV.2, Exercise 2]{gelfandMethodsHomologicalAlgebra}
        \[
            E_2^{p,q}
            =\bigoplus_{i\in \bbZ}\Hom{\!}^p(\HH^ix,\HH^{i+q}x)
            \Longrightarrow\Hom{\!}^{p+q}(x,x)
        \]
        and inductively show enough rows vanish on the second page. Suppose we
        have shown that the terms \(E_2^{p,q}\) all vanish for \(q\leq j\), from
        the above discussion this is true for \(j=-n\). Since the terms also
        vanish if \(p\leq 0\), we see that the differentials in and out of
        \(E_2^{0,1+j}\) and \(E_2^{1,1+j}\) are trivial so that these terms
        persist in future pages of the sequence, arising as subquotients of
        \(\Hom^{1+j}(x,x)\) and \(\Hom^{2+j}(x,x)\) respectively.

        In particular if \(j<-2\), then \(E_2^{0,1+j}=E_2^{1,1+j}=0\) and we
        have
        \begin{align*}
            \Hom{\!}^{\phantom{1}}(\HH^0x,\HH^{1+j}x)
            =\Hom{\!}^{\phantom{1}}(\HH^{-n-j-1}x,\HH^{-n}x) =0,\\
            \Hom{\!}^1(\HH^0x,\HH^{1+j}x)
            =\Hom{\!}^1(\HH^{-n-j-1}x,\HH^{-n}x)=0.
        \end{align*}
        Since \(\HH^0(x)\) and
        \(\HH^{-n}(x)\) both lie in \(\langle \OO_p\;|\;p\in X\rangle\),
        \cref{lem:skyscraperdetection} and the above equalities show that the
        same holds for \(\HH^{1+j}(x)\) and \(\HH^{-n-j-1}(x)\). In fact an
        identical reasoning for lower rows shows \(\HH^i(x)\in \langle
        \OO_p\;|\;p\in X\rangle\) for each of the values
        \(i\in\{-n,-n+1,...,j+1\}\cup \{-n-j-1,-n-j,...,0\}\). For these values
        of \(i\), the vanishing of \(\Hom(\HH^ix,\HH^{i+j+1}x)\subset
        E_2^{0,1+j}\) then shows that
        \(\HH^i(x)\) and \(\HH^{i+j+1}(x)\) have disjoint supports, so that
        \(E_2^{p,j+1}=0\) for remaining values of \(p\) too.

        Thus for each \(q\leq -2\), we have \(E_2^{0,q}=E_2^{1,q}=0\), and hence
        \(\HH^q(x)\) and \(\HH^{-n-q}(x)\) have zero--dimensional supports as
        above.

        If \(n\geq 3\) this
        accounts for all cohomologies of \(x\). If \(n=2\) the above argument
        fails to address the cohomology in degree \(-1\), but note that in this
        case the vanishing of
        \(E_2^{0,-1}=\Hom(\HH^0x,\HH^{-1}x)\oplus\Hom(\HH^{-1}x,\HH^{-2}x)\)
        alone suffices to apply \cref{lem:skyscraperdetection} and deduce that
        \(\dim(\Supp(\HH^{-1}x))=0\).
    \end{proof}

    \begin{proof}[Proof that cohomologies have mutually disjoint supports] Since
        \(x\) is an extension of shifts of skyscrapers, it admits a
        decomposition \(x=\bigoplus_{p\in C}x|_p\) indexed over closed points of
        \(C\), where the summand \(x|_p\) is a complex supported within
        \(\{p\}\). In particular the top non--zero cohomology of \(x|_p\) admits
        a non-trivial morphism (which factors via \(\OO_p\)) to the bottom
        non--zero cohomology of \(x|_p\).
        But \(x|_p\) also satisfies \(\Hom^{<0}(x|_p,x|_p)=0\), whence \(x|_p\)
        must be a sheaf concentrated in a single cohomological degree. Since
        \(\HH^i(x)=\bigoplus_{p\in C}\HH^i(x|_p)\), the result follows.
    \end{proof}
\end{lemma}

The above result, combined with \cref{lem:inductioncontinues} and
\cref{cor:twotermclassification}, finishes our analysis of all complexes
supported on \(C\) that have only non--negative self extensions, describing a
complete proof of \cref{keythm:objclassification}.
\pagebreak

\section{Corollaries, corollaries}

A \emph{brick} in \(\Dm X\) is a complex \(x\)  whose self--extensions mimic
those of a simple module, i.e.\ \(\Hom^{<0}(x,x)=0\) and \(\Hom(x,x)=\bbC\). A
\emph{semibrick} is a set \(\{x_i\;|\;i\in I\}\) of bricks that satisfy the
orthogonality \(\Hom^{\leq 0}(x_i,x_j)=0\) whenever \(i\neq j\). A
\emph{simple--minded collection} is a finite semibrick that is not contained in
a proper thick subcategory of \(\Dm X\). The heart of an algebraic t-structure
is the extension--closure of a simple--minded collection, conversely every
simple--minded collection is the set of simples of an algebraic heart \cite{al-nofayeeSimpleObjectsHeart}.

We sketch how finite semibricks, hence also bricks, simple--minded collections, and
algebraic t-structures in \(\Dm X\), are all readily classified using
\cref{keythm:objclassification}. Throughout, the phrase \emph{up to mutation}
means up to shifts and iterated application of functors appearing in
\eqref{eqn:allfunctors} and their inverses.

\begin{corollary}
    \label{cor:sbrickclassification}
    Every finite semibrick \(\{x_0,...,x_m\}\subset \Dm X\) must, up to
    mutation, either be a collection of simple \(\Lambda_i-\)modules for some \(i\in
    \bbZ/N\bbZ\), or be of the form \(\big\{\OO_{p_0}[n_0],...,\OO_{p_m}[n_m]\big\}\)
    for distinct closed points \(p_0,...,p_m\) in \(X\) (or in \(W\)) and
    integers \(n_0,...,n_m\in \bbZ\).

    Consequently every brick in \(\Dm X\) is, up to mutation, a simple
    \(\Lambda_i-\)module for some \(i\in \bbZ/N\bbZ\), or a skyscraper sheaf
    \(\OO_p\) at a closed point in \(X\) or in \(W\). Likewise, any heart of an
    algebraic t-structure in \(\Dm X\) is, up to mutation, the standard heart
    \(\rflmod\Lambda_i\subset \Dfl\Lambda_i\) for some \(i\in \bbZ/N\bbZ\).

    \begin{proof}
        Consider the object \(x=x_0\oplus ... \oplus x_m\), which clearly satisfies
        \(\Hom^{<0}(x,x)=0\) and is therefore classified by
        \cref{keythm:objclassification}. If it fits the description
        \ref{item:objclass2}, then each \(x_i\) must be a shift of some sheaf
        \(\OO_p\) (\(p\in X\)) since those are the only brick sheaves with
        zero--dimensional support (to see this, note all objects in \(\langle
        \OO_p\rangle\) have a non-zero endomorphism which factors via \(\OO_p\)).

        On the other hand if \(x\) fits the description \ref{item:objclass1},
        i.e.\ the semibrick \(\{x_0,...,x_m\}\) up to mutation lies in some
        \(H'=\rflmod\Lambda_j\), then consider the smallest torsion class
        \(T\subset H'\) that contains \(x\). By \cite[proposition
        2.6]{shimpiTorsionPairs3fold}, each \(x_i\) is a simple object in the
        heart \(K\) obtained by tilting \(H'\) in this torsion class. Examining
        the Hasse quiver of \(\big[H',H'[1]\big]\) as sketched in
        \cref{subsec:dreaminduction}, we see that \(K\) must also be a tilt of
        \(\Phi(\rflmod\Lambda_0)[1]\) or \(\Phi^{-1}(\rflmod\Lambda_0)\) for
        some mutation functor \(\Phi\). Thus \(K\) is, up to mutation, a tilt of
        \(\rflmod\Lambda_0\) and required descriptions of the simple objects
        \(x_0,...,x_m\in K\) can be read off from \cref{thm:tiltclassification}.

        To address the last case, suppose \(x\) is a two--term complex of
        coherent sheaves as in \ref{item:objclass3}, without loss of generality
        on \(X\). Each summand \(x_i\) must then also admit the same
        description, sitting in an exact triangle \[\mathcal{F}_i[1]\to x_i\to
        \mathcal{G}_i\to \mathcal{F}_i[2]\] with \(\mathcal{G}_i\in \langle
        \OO_p\;|\;p\in X\rangle\), \(\mathcal{F}_i\in \coh X\), and
        \(\Hom(\mathcal{G}_i,\mathcal{F}_i)=0\).

        Claim for each \(i\), either \(\mathcal{F}_i\) or \(\mathcal{G}_i\)
        vanishes. Indeed if \(\mathcal{F}_i\neq 0\), it cannot be a summand in
        \(\mathcal{G}_i[-2]\) so the first map in the
        triangle above is a non--zero element in \(\Hom(\mathcal{F}_i[1],x_i)\).
        This element is annihilated by the natural map
        \(\Hom(\mathcal{F}_i[1],x_i)\to \Hom(\mathcal{G}_i[-1],x_i)\) since any two
        consecutive maps in an exact triangle compose to zero
        \cite[\href{https://stacks.math.columbia.edu/tag/0146}{tag
        0146}]{stacks-project}. Thus in the long exact sequence
        \[
            ...\to\underbrace{\Hom(\mathcal{F}_i[2],x_i)}_{0}\longrightarrow\Hom(\mathcal{G}_i,x_i)\longrightarrow
            \underbrace{\Hom(x_i,x_i)}_{\mathbb{C}}\xrightarrow{\;\;\alpha\;\;}
            \Hom(\mathcal{F}_i[1],x_i)\longrightarrow
            \Hom(\mathcal{G}_i[-1],x_i)\to ...
        \]
        the final map has non-trivial kernel so that \(\alpha\) is injective. It
        follows that \(\Hom(\mathcal{G}_i,x_i)=0\), and considering the
        long exact sequence associated to \(\Hom(\mathcal{G}_i,-)\) shows
        \(\Hom(\mathcal{G}_i,\mathcal{F}_i[1])\) vanishes too. If
        \(\mathcal{G}_i\) were non--zero then \cref{lem:skyscraperdetection}
        would show \(\Supp(\mathcal{F}_i)\) is zero--dimensional hence disjoint
        from \(\Supp(\mathcal{G}_i)\), but this is absurd since we would then
        have a non--trivial splitting
        \(x_i=\mathcal{F}_i[1]\oplus\mathcal{G}_i\) of a brick.

        Thus there is a \(k\geq 0\) and a reordering of the summands of \(x\) so
        that \(x_i=\mathcal{G}_i\) for \(i < k\), and \(x_i=\mathcal{F}_i[1]\)
        for \(i\geq k\). If \(k\geq 1\) then \cref{lem:easyskyscraperdetection}
        shows each \(\mathcal{F}_i\) (\(i\geq k\)) has zero--dimensional support
        since \(\Ext^{-1}(x_0,x_i)=\Hom(\mathcal{G}_0,\mathcal{F}_i)\) and
        \(\Hom(x_0,x_i)=\Ext^1(\mathcal{G}_0,\mathcal{F}_i)\) both vanish, whence \(x\)
        also admits the description \ref{item:objclass2} and can be addressed as
        in the first paragraph.

        On the other hand if \(k=0\) (so that \(x\in
        \coh X[1]\)), then we can consider (as in \cref{prop:tiltingimproves})
        the maximal tilt \(\ol K\) of \(H=\rflmod\Lambda_0\) satisfying \(x\in
        \ol K\). We have \(\ol K \geq \coh X[1]\)  (i.e.\ \(\ol K\cap H
        \subseteq \coh X[1]\cap H\)), and equality cannot hold since the torsion
        class \(\coh X[1]\cap H\) is the union of an increasing chain in
        \(\tors(H)\) while the torsion class \(\ol K\cap H\) is the intersection
        of all torsion classes containing \(\HH^0(x)\) (see proof of
        \cref{lem:collapse}). It follows from \cref{thm:tiltclassification} that
        \(\ol K\) is an algebraic heart, so that \(x\in \ol K\) also admits the
        description \ref{item:objclass1} and is addressed by the second
        paragraph.
    \end{proof}
\end{corollary}
\medskip

The classification of algebraic t-structures in \(\Dm X\) can be lifted to a
classification of tilting complexes in the bigger category \(\Db X\), following
\cite[\S 5.4]{hiranoStabilityConditions3fold}. Recall that a complex \(t\in \Db
X\) is \emph{tilting} if it satisfies \(\Hom^i(t,t)=0\) for all \(i\neq 0\) and
the smallest thick subcategory of \(\Db X\) containing \(t\) also contains all
perfect complexes.

Given an algebra \(\Lambda\) and a derived equivalence \(\Psi:\Db \Lambda\to \Db
X\), the image \(\Psi(\Lambda)\) of the regular \(\Lambda\)--module is a tilting
object in \(\Db X\). The converse is Rickard's theorem \cite[theorem
6.4]{rickardMoritaTheoryDerived}, that any tilting object \(t\in \Db X\) arises
in the above fashion from the \(R\)--algebra \(\Lambda=\End(t)\) and the
\(R\)--linear equivalence determined as \(\Psi^{-1}=\RHom(t,-)\). If the complex
\(t\) (equivalently the algebra \(\Lambda\)) is \emph{basic}, i.e.\ if distinct
indecomposable summands of \(t\) are non--isomorphic, then Morita
theory\footnote{This requires \(\Lambda\) to be semi--perfect, which is true
since it is a module--finite algebra over the complete local domain \(R\).}
\cite[see e.g.][proposition 18.37]{lamLecturesModulesRings} picks out \(t\) as
the unique basic projective generator in the heart \(\Psi(\rmod\Lambda)\subset
\Db X\).

In this fashion, each equivalence \(\Db \Lambda_i\to \Db X\) composed of the
functors appearing in \eqref{eqn:allfunctors} and their inverses gives a unique
basic tilting object in \(\Db X\). These tilting
objects sit in a connected tetravalent exchange graph so that each one can be
obtained from \(\text{VdB}^{-1}(\Lambda_0)=\mathscr{N}\oplus\OO_X\) by iterated
left or right mutation in indecomposable summands (i.e.\ by \emph{simple
mutation}). Other tilting complexes include non--trivial shifts of those already
considered, these can be obtained by simple mutation from \((\mathscr{N}\oplus
\OO_X)[n]\). We show this list is exhaustive.

\begin{corollary}
    \label{cor:tiltingclassification}
    Any basic tilting complex \(t\in \Db X\) can be obtained by iteratively
    mutating the tilting complex \((\mathscr{N}\oplus\OO_X)[n]\) in
    indecomposable summands, for some \(n\in \bbZ\).
\end{corollary}

\begin{proof}
    Consider the module--finite \(R\)--algebra \(\Lambda=\End(t)\) and the
    induced equivalence \(\Psi:\Db \Lambda\to \Db X\) with inverse
    \(\RHom(t,-)\). By \(R\)--linearity, the functor \(\Psi\) restricts to an
    equivalence \(\Psi:\Dfl\Lambda\to \Dm X\) and hence the image of
    \(\rflmod\Lambda\) is an algebraic heart in \(\Dm X\). By
    \cref{cor:sbrickclassification}, we thus see that \(\Psi(\rflmod\Lambda)\)
    agrees with some heart \(\Phi(\rflmod\Lambda_i)[n]\) for \(i\in
    \bbZ/N\bbZ\), \(n\in \bbZ\), and \(\Phi\) a composite of functors appearing in
    \eqref{eqn:allfunctors} and their inverses.

    But by \cite[lemma 5.7]{hiranoStabilityConditions3fold}, the
    subcategory \(\rflmod\Lambda\subset \Db\Lambda\) determines the natural
    t-structure via
    \[
        \rmod \Lambda\,[\geq 0] = \left\{x\in \Db\Lambda \;|\;
        \Hom{\!}^j(x,y)=0\text{ for all } j<0 \text{ and all }y\in
        \rflmod\Lambda\right\}.
    \]
    Likewise the subcategory \(\rflmod\Lambda_i\subset \Db \Lambda_i\)
    determines the natural t-structure with heart \(\rmod\Lambda_i\), thus we
    see that the the equality \(\Psi(\rflmod\Lambda)=\Phi(\rflmod\Lambda_i)[n]\)
    extends to an equality \(\Psi(\rmod\Lambda)=\Phi(\rmod\Lambda_i)[n]\). In
    particular \(t=\Psi(\Lambda)\) must equal the basic projective generator
    \(\Phi(\Lambda_i)[n]\in \Phi(\rmod\Lambda_i)[n]\), so that \(t\) is obtained
    by iterated simple mutation from \((\mathscr{N}\oplus \OO_X)[n]\) as required.
\end{proof}

Note that the shift \(n\) appearing above is unique, and can be read off
as the cohomological degree of the generic stalk \(\bbk\otimes_R t\) where
\(\bbk=\bbC(Z)\) is the function field of \(Z\). Indeed considering
\(t\in\Db\Lambda_0\) (across Van den Bergh's equivalence), the localisation
\(\bbk\otimes t\) is a tilting module over \(\bbk\otimes\Lambda_0\) \cite[lemma
6.4]{namanyaGroupActionsAlgebraic}. But \(\bbk\otimes \Lambda_0\) is a matrix
ring over the field \(\bbk\) (see lemma 6.2 \emph{ibid}) so every tilting
\((\bbk\otimes\Lambda_0)\)--module is the shift of a \(\bbk\)--vector space. It
follows that \(t\) is obtained by simple mutation from \(\Lambda_0[n]\) if and
only if \(\bbk\otimes t[-n]=\bbk\otimes\Lambda_0\) is a \(\bbk\)--vector space.
\medskip

We conclude by addressing the connectivity of the stability manifold of \(\Dm
X\), referring the reader to \cite[\S 3]{bridgelandSpacesStabilityConditions}
for definitions and generalities. The stability manifold \(\Stab(\Dm X)\) has been extensively
analysed
\cite{bridgelandStabilityConditionsKleinian,todaStabilityConditionsCrepant,hiranoStabilityConditions3fold}
with a key result being the existence of a connected component \(\Stab^\circ
(\Dm X)\) that contains all stability conditions \((Z,P)\) for which some
heart \(P(\varphi,\varphi+1]\) agrees with some \(\rflmod\Lambda_j\) up to
mutation, equivalently is algebraic (\cref{cor:sbrickclassification}). We show
that every locally finite stability condition on \(\Dm X\) is of this form.

\begin{corollary}
    \label{cor:stabconnected}
    For any locally finite Bridgeland stability condition \((Z,P)\) on \(\Dm
    X\), there is a phase \(\varphi\in \bbR\) such that the heart
    \(P(\varphi,\varphi+1]\) is algebraic. In particular the stability
    manifold \(\Stab(\Dm X)=\Stab^\circ(\Dm X)\) is connected.

    \begin{proof}
        Given the stability condition \(\sigma=(Z,P)\), consider the set
        \(\big\{\varphi\in \bbR\;\big\vert\; P(\varphi)\neq \emptyset\big\}\) of
        phases occupied by \(\sigma\)--stable objects. Now a \(\sigma\)--stable
        object \(x\in \Dm X\) of phase \(\varphi\) is necessarily simple in the
        Abelian category \(P(\varphi)\) and therefore is a brick, equal to a
        simple \(\Lambda_i\)--module or a skyscraper sheaf up to mutation
        (\cref{cor:sbrickclassification}).

    \begin{minipage}[c]{0.67\textwidth}
        But \cite[lemma 4.18]{shimpiTorsionPairs3fold} then shows that the
        \(\KK\)--theory class \([x]\) is a restricted affine root in the rank
        two lattice \(\KK(\Dm X)\), that is to say there is a root system
        \(\Pi \subset \bbZ^{\oplus m}\) associated to a Cartan matrix of affine
        type, and a projection \(\bbZ^{\oplus m}\onto \bbZ^{\oplus 2}\simeq
        \KK(\Dm X)\) which maps a root \(\alpha\in \Pi\) to \([x]\). Such
        restricted root systems have been studied in \cite[chapter
        2]{iyamaTitsConeIntersections}, they define affine hyperplane
        arrangements in \(\bbR^2\) that are locally finite away from the origin.
        It follows that the possible phases of the complex number \(Z[x]\) must
        be restricted to a countable set of values with a discrete set of
        accummulation points.

        The example alongside is the affine \(\text{D}_4\) restricted root
        system associated to a length \(2\) flop --- clearly all phases
        accummulate towards that of \(\pm \delta\).
    \end{minipage}\hspace{0.03\textwidth}
    \begin{minipage}[c]{0.3\textwidth}
        \includegraphics[width=\textwidth]{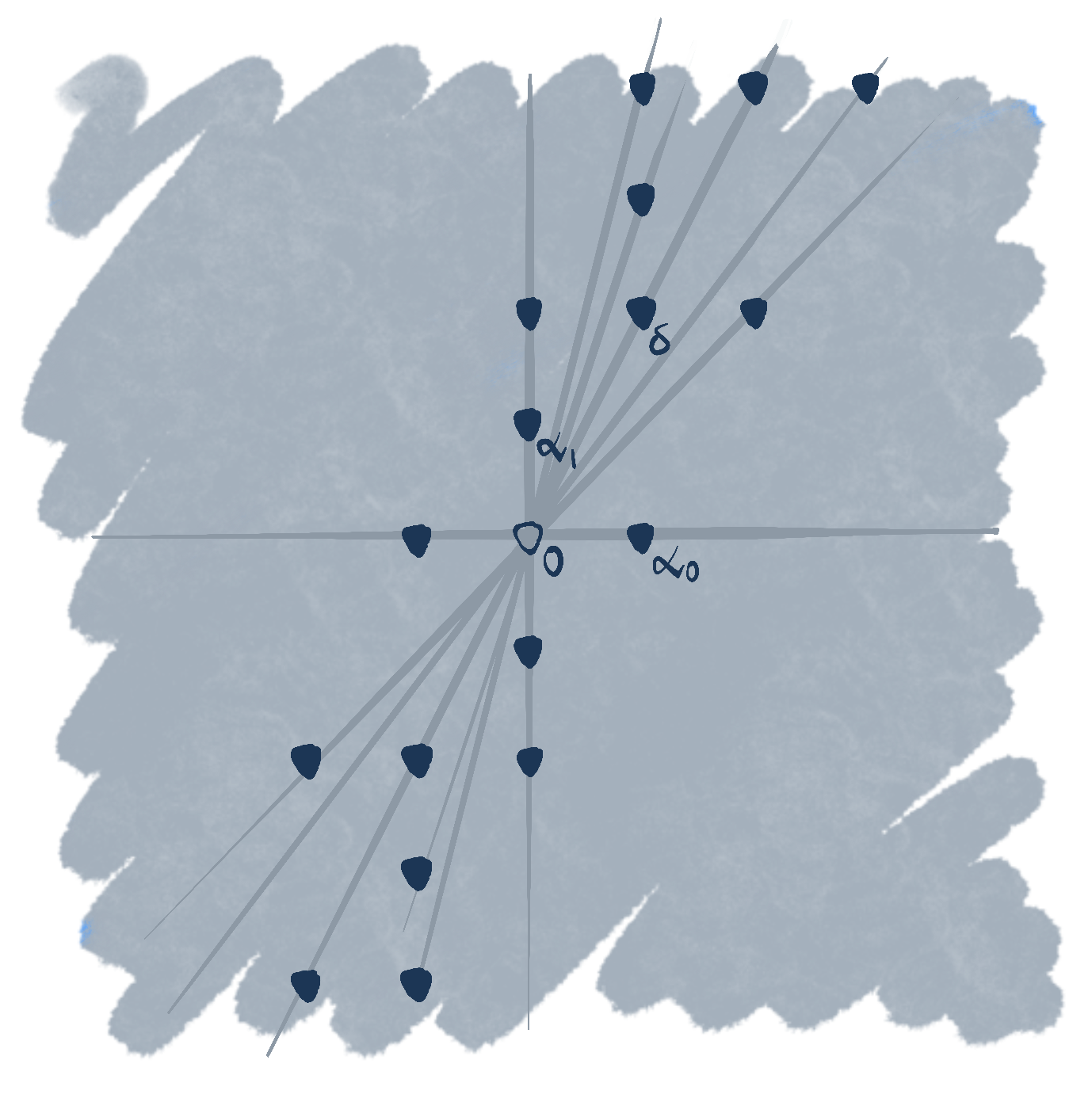}
    \end{minipage}

        In particular \(\sigma\) has a
        `non-trivial phase gap', i.e.\ there is a \(\varphi\in \bbR\) and
        \(\epsilon>0\) such that \(P(\varphi,\varphi+\epsilon)=\emptyset\). If
        the image of \(Z:\KK(\Dm X)\to \bbC\) is a discrete subgroup of
        \(\bbC\), then we may use
        \cite[lemma 4.4]{bridgelandStabilityConditionsK3} to conclude that
        \(P(\varphi,\varphi+1]=P(\varphi+\epsilon,\varphi+1]\) is an algebraic
        Abelian category. On the other hand if \(\text{img}(Z)\) is a
        non--discrete rank \(2\) subgroup of \(\bbC\), then this image must
        lie in the \(\bbR\)--span of some non--zero complex number
        \(e^{i\pi\varphi}\) so that the phase of each stable object is an
        integer away from \(\varphi\). Then local finiteness allows us to
        conclude that \(P(\varphi,\varphi+1]=P(\varphi+1)\) is algebraic.
    \end{proof}
\end{corollary}

\fakesection{References}
\printbibliography

@article{al-nofayeeSimpleObjectsHeart,
  title = {Simple Objects in the Heart of a T-Structure},
  author = {{Al-Nofayee}, Salah},
  year = {2009},
  journal = {Journal of Pure and Applied Algebra},
  volume = {213},
  number = {1},
  doi = {10.1016/j.jpaa.2008.05.014},
}

@article{aiharaClassifyingTiltingComplexes,
  title = {Classifying Tilting Complexes over Preprojective Algebras of {{Dynkin}} Type},
  author = {Aihara, Takuma and Mizuno, Yuya},
  year = {2017},
  journal = {Algebra \& Number Theory},
  volume = {11},
  number = {6},
  archiveprefix = {arXiv},
  eprint = {1509.07387},
}

@article{aspinwallPointsPointView,
  title = {A {{Point}}'s {{Point}} of {{View}} of {{Stringy Geometry}}},
  author = {Aspinwall, Paul S.},
  year = {2003},
  journal = {Journal of High Energy Physics},
  volume = {2003},
  number = {01},
  archiveprefix = {arXiv},
  eprint = {hep-th/0203111},
}

@article{augustFinitenessDerivedEquivalence,
  title = {On the {{Finiteness}} of the {{Derived Equivalence Classes}} of Some {{Stable Endomorphism Rings}}},
  author = {August, Jenny},
  year = {2020},
  journal = {Mathematische Zeitschrift},
  volume = {296},
  number = {3-4},
  archiveprefix = {arXiv},
  eprint = {1801.05687},
}

@article{bridgelandFlopsDerivedCategories,
  title = {Flops and Derived Categories},
  author = {Bridgeland, Tom},
  year = {2002},
  journal = {Inventiones Mathematicae},
  volume = {147},
  number = {3},
  archiveprefix = {arXiv},
  eprint = {math/0009053},
}

@article{bridgelandStabilityConditionsK3,
  title = {Stability Conditions on {{K3}} Surfaces},
  author = {Bridgeland, Tom},
  year = {2008},
  journal = {Duke Mathematical Journal},
  volume = {141},
  number = {2},
  archiveprefix = {arXiv},
  eprint = {math/0307164},
}

@misc{bridgelandSpacesStabilityConditions,
  title = {Spaces of Stability Conditions},
  author = {Bridgeland, Tom},
  year = {2006},
  number = {arXiv:math/0611510},
  publisher = {arXiv},
  archiveprefix = {arXiv},
  eprint = {math/0611510},
}

@article{bridgelandStabilityConditionsKleinian,
  title = {Stability Conditions and {{Kleinian}} Singularities},
  author = {Bridgeland, Tom},
  year = {2009},
  journal = {International Mathematics Research Notices},
  archiveprefix = {arXiv},
  eprint = {math/0508257},
}

@article{donovanStringyKahlerModuli,
  title = {Stringy {{K{\"a}hler}} Moduli, Mutation and Monodromy},
  author = {Donovan, Will and Wemyss, Michael},
  year = {2025},
  journal = {Journal of Differential Geometry},
  volume = {129},
  number = {1},
  archiveprefix = {arXiv},
  eprint = {1907.10891},
}

@book{gelfandMethodsHomologicalAlgebra,
  title = {Methods of Homological Algebra},
  author = {Gel'fand, S. I. and Manin, Y},
  year = {1996},
  publisher = {Springer},
  langid = {english},
}

@book{happelTiltingAbelianCategories,
  title = {Tilting in {{Abelian}} Categories and Quasitilted Algebras},
  author = {Happel, Dieter and Reiten, Idun and Smal{\o}, Sverre O.},
  year = {1996},
  series = {Memoirs of the {{American Mathematical Society}}},
  number = {no. 575},
  publisher = {American Mathematical Soc.},
  lccn = {QA3 QA251.5 .A57 no. 575},
}

@article{hiranoStabilityConditions3fold,
  title = {Stability Conditions for 3-Fold Flops},
  author = {Hirano, Yuki and Wemyss, Michael},
  year = {2023},
  journal = {Duke Mathematical Journal},
  volume = {172},
  number = {16},
  archiveprefix = {arXiv},
  eprint = {1907.09742}
}

@article{haraSphericalObjectsDimensions,
  title = {Spherical Objects in Dimensions Two and Three},
  author = {Hara, Wahei and Wemyss, Michael},
  year = {2024},
  journal = {Journal of the European Mathematical Society},
  archiveprefix = {arXiv},
  eprint = {2205.11552}
}

@article{ishiiAutoequivalencesDerivedCategories,
  title = {Autoequivalences of Derived Categories on the Minimal Resolutions of {{An-singularities}} on Surfaces},
  author = {Ishii, Akira and Uehara, Hokuto},
  year = {2005},
  journal = {J. of Differential Geometry},
  volume = {71},
  number = {3},
  archiveprefix = {arXiv},
  eprint = {math/0409151}
}

@misc{iyamaTitsConeIntersections,
  title = {Tits Cone Intersections and Applications},
  author = {Iyama, Osamu and Wemyss, Michael},
  url = {https://www.maths.gla.ac.uk/~mwemyss/MainFile_for_web.pdf},
  note = {(In preparation)},
}

@article{keatingSymplectomorphismsSphericalObjects2024,
  title = {Symplectomorphisms and Spherical Objects in the Conifold Smoothing},
  author = {Keating, Ailsa and Smith, Ivan},
  year = {2024},
  journal = {Compositio Mathematica},
  volume = {160},
  number = {11},
  archiveprefix = {arXiv},
  eprint = {2301.10525}
}

@article{kapranovKleinianSingularitiesDerived,
  title = {Kleinian Singularities, Derived Categories and {{Hall}} Algebras},
  author = {Kapranov, M. and Vasserot, E.},
  year = {2000},
  journal = {Mathematische Annalen},
  volume = {316},
  number = {3},
  archiveprefix = {arXiv},
  eprint = {math/9812016}
}

@book{lamLecturesModulesRings,
  title = {Lectures on Modules and Rings},
  author = {Lam, T. Y.},
  year = {1999},
  series = {Graduate Texts in Mathematics},
  number = {189},
  publisher = {Springer},
  lccn = {QA247 .L263 1999},
}

@article{rickardMoritaTheoryDerived,
  title = {Morita {{Theory}} for {{Derived Categories}}},
  author = {Rickard, Jeremy},
  year = {1989},
  journal = {Journal of the London Mathematical Society},
  volume = {s2-39},
  number = {3},
}

@misc{robertsContractibleStability,
  title = {The topology of 2-dimensional stability spaces},
  author = {Roberts, Stefan and Woolf, Jon},
  note = {(In preparation)},
  shorthand = {RW\({}^+\)}
}

@misc{shimpiTorsionPairs3fold,
  title = {Torsion Pairs and 3-Fold Flops},
  author = {Shimpi, Parth},
  year = {2025},
  number = {arXiv:2502.05146},
  eprint = {2502.05146},
  publisher = {arXiv},
  archiveprefix = {arXiv},
}

@article{seidelBraidGroupActions,
  title = {Braid Group Actions on Derived Categories of Coherent Sheaves},
  author = {Seidel, Paul and Thomas, Richard},
  year = {2001},
  journal = {Duke Mathematical Journal},
  volume = {108},
  number = {1},
  archiveprefix = {arXiv},
  eprint = {math/0001043}
}

@article{todaStabilityConditionsCrepant,
  title = {Stability Conditions and Crepant Small Resolutions},
  author = {Toda, Yukinobu},
  year = {2008},
  journal = {Transactions of the Americal Mathematical Society},
  volume = {360},
  number = {11},
  archiveprefix = {arXiv},
  eprint = {math/0512648}
}

@article{vandenberghThreedimensionalFlopsNoncommutative,
  title = {Three-Dimensional Flops and Noncommutative Rings},
  author = {{Van den Bergh}, Michel},
  year = {2004},
  journal = {Duke Mathematical Journal},
  volume = {122},
  number = {3},
  publisher = {Duke University Press},
  shorthand= {VdB04},
  archiveprefix = {arXiv},
  eprint = {math/0207170}
}

@article{wemyssFlopsClustersHomological,
  title = {Flops and {{Clusters}} in the {{Homological Minimal Model Programme}}},
  author = {Wemyss, Michael},
  year = {2018},
  journal = {Inventiones Mathematicae},
  volume = {211},
  number = {2},
  archiveprefix = {arXiv},
  eprint = {1411.7189}
}

@misc{namanyaGroupActionsAlgebraic,
  ids = {namanyaGroupActionsAlgebraica,namanyaGroupActionsAlgebraicb},
  title = {Group {{Actions}} from {{Algebraic Flops}}},
  author = {Namanya, Caroline},
  year = {2023},
  number = {arXiv:2310.18062},
  eprint = {2310.18062},
  publisher = {arXiv},
  archiveprefix = {arXiv},
}

@online{stacks-project,
  author       = {The {Stacks project authors}},
  title        = {The Stacks project},
  url          = {https://stacks.math.columbia.edu},
  year         = {2025},
  shorthand    = {Stacks}
}

@book{matsumuraCommutativeRingTheory,
  title = {Commutative Ring Theory},
  author = {Matsumura, Hideyuki},
  date = {2008},
  series = {Cambridge Studies in Advanced Mathematics},
  number = {8},
  publisher = {Cambridge Univ. Press},
  location = {Cambridge},
  isbn = {978-0-521-36764-6},
}
{\footnotesize%
	\textsc{The Mathematics and Statistics Building, University
		of Glasgow, University Place, Glasgow G12 8QQ, UK.}\\
	\textit{Email address: }\texttt{parth.shimpi@glasgow.ac.uk}\\
	\textit{Web: }\texttt{https://pas201.user.srcf.net}}
\end{document}